\documentclass[a4paper,reqno]{amsart}

\usepackage[utf8]{inputenc}
\usepackage[english]{babel} 

\usepackage{amsmath,amssymb,amsfonts,amsthm}
\usepackage{mathrsfs, esint, dsfont}
\usepackage{moreverb,rotating,graphics,float, subcaption}
\usepackage[colorlinks, linkcolor={blue!50!blue}, citecolor={red}]{hyperref}
\usepackage{framed,fancybox}
\usepackage{enumerate}
\usepackage{csquotes}
\usepackage{float}
\usepackage{todonotes}

\newtheorem{theorem}{Theorem}
\newtheorem{proposition}[theorem]{Proposition}
\newtheorem{remark}[theorem]{Remark}
\newtheorem{lemma}[theorem]{Lemma}
\newtheorem{corollary}[theorem]{Corollary}
\newtheorem{definition}[theorem]{Definition}

\newtheorem{conjecture}[theorem]{Conjecture}

\newcommand\1{{\mathds 1}}

\def\C{{\mathbb C}}

\def\bbI{{\mathbb I}}

\def\N{{\mathbb N}}

\def\R{{\mathbb R}}

\def\SS{{\mathbb S}}

\def\Z{{\mathbb Z}}

\def\rd{{\mathrm{d}}}
\def\re{{\mathrm{e}}}
\def\ri{{\mathrm{i}}}

\def\cE{{\mathcal E}}
\def\cF{{\mathcal F}}

\def\cN{{\mathcal N}}

\def\cR{{\mathcal R}}

\def\gS{{\mathfrak{S}}}

\newcommand{\eps}{\epsilon}

\newcommand{\ii}{\infty}
\newcommand{\dps}{\displaystyle}
\newcommand{\nn}{\nonumber}
\newcommand{\Tr}{{\rm Tr}}
\newcommand{\VTr}{\underline{\rm Tr}}
\newcommand{\tr}{{\rm Tr}}

\newcommand{\Lat}{\cR}
\newcommand{\RLat}{{\cR^*}}
\newcommand{\WS}{\mathbb{K}}

\newcommand{\bra}{\langle}
\newcommand{\ket}{\rangle}

\renewcommand{\epsilon}{\varepsilon}
\newcommand\pscal[1]{{\ensuremath{\left\langle #1 \right\rangle}}}
\newcommand{\norm}[1]{ \left| \! \left| #1 \right| \! \right| }

\author{David Gontier}
\address{CEREMADE, University of Paris-Dauphine, PSL University, 75016 Paris, France}
\email{gontier@ceremade.dauphine.fr}

\author{Mathieu Lewin}
\address{CNRS and CEREMADE, University of Paris-Dauphine, PSL University, 75016 Paris, France}
\email{mathieu.lewin@math.cnrs.fr}

\author{Faizan Q. Nazar}
\address{CEREMADE, University of Paris-Dauphine, PSL University, 75016 Paris, France}
\email{nazar@ceremade.dauphine.fr}

\title[NLS equation for orthonormal functions I]{The nonlinear Schr\"odinger equation for orthonormal functions:\\ I. Existence of ground states}

\date{\today}

\begin{document}

\begin{abstract}
We study the nonlinear Schr\"odinger equation for systems of $N$ orthonormal functions. We prove the existence of ground states for all $N$ when the exponent $p$ of the non linearity is not too large, and for an infinite sequence $N_j$ tending to infinity in the whole range of possible $p$'s, in dimensions $d\geq1$. This allows us to prove that translational symmetry is broken for a quantum crystal in the Kohn-Sham model with a large Dirac exchange constant. 

\bigskip

\noindent \sl \copyright~2020 by the authors. This paper may be reproduced, in its entirety, for non-commercial purposes.
\end{abstract}

\maketitle

\section{Introduction and statement of the main results}

\subsection{Ground states for the nonlinear Schr\"odinger equation}

The \emph{nonlinear Schr\"odinger equation} (NLS) is one of the most famous nonlinear partial differential equation and it naturally occurs in a variety of physical situations~\cite{Malomed-05}, including Bose-Einstein condensation~\cite{PitStr-03, LieSeiSolYng-05}, nonlinear optics~\cite{ZakSha-72,Manakov-74}, water waves~\cite{Zakharov-68}, Langmuir waves in plasmas~\cite{TanYaj-69,FriIch-73} and many others. We quickly recall here some of its mathematical properties before turning to its generalisation to systems of orthonormal functions, which is appropriate for quantum mechanical systems. 

An \emph{NLS ground state} is by definition a normalised positive solution $Q\in L^2(\R^d,\R_+)$ to the stationary focusing NLS equation in $\R^d$:
\begin{equation}
\left(-\Delta-Q^{2p-2}\right)Q=\mu\,Q, \qquad
\int_{\R^d} Q(x)^2\rd x = 1,\qquad \mu<0.
\label{eq:NLS}
\end{equation}
In this formulation, the Lagrange multiplier $\mu$ is unknown and must be adjusted to fulfill the normalisation constraint in $L^2(\R^d)$. For
\begin{equation}
1<p<\begin{cases}\ii & \text{in dimensions $d=1,2$,}\\
\dps\frac{d}{d-2} & \text{in dimensions $d\geq3$.}
     \end{cases}
     \label{eq:intro_constraint_p_subcritical}
\end{equation}
it is known~\cite{Coffman-72,Kwong-89,McLeod-93,Tao-06,Frank-13} that~\eqref{eq:NLS} admits a \emph{unique} solution $(Q,\mu)$, up to space translations for the function $Q$. In fact, $Q$ is a certain dilation of the unique solution to the equation with $\mu=-1$ and without the mass constraint. If $d\geq3$ and $p \ge \frac{d}{d-2}$, no solution with $\mu<0$ can exist~\cite{Pohozaev-65,BerLio-83}. At the critical exponent $p = \frac{d}{d-2}$ there is also a unique solution of the equation but it has $\mu=0$ (which results in an additional invariance under dilations) and it belongs to $L^2(\R^d)$ only in dimensions $d\geq5$. When $p$ is as in~\eqref{eq:intro_constraint_p_subcritical}, the unique solution is non-degenerate~\cite{Weinstein-85}, which plays an important role for the behaviour of the associated time-dependent Schr\"odinger equation
$$i\frac\partial{\partial t}u=\left(-\Delta-|u|^{2p-2}\right)u,$$
of which $Q$ is a stationary state. Since $Q>0$ with $\int_{\R^d} Q(x)^2 \, \rd x = 1$, the Lagrange multiplier $\mu$ must necessarily be the lowest eigenvalue of the operator $-\Delta-Q^{2p-2}$, see~\cite[Cor.~11.9]{LieLos-01} and~\cite[Sec.~12]{ReeSim4}. 
By scaling we find that the (unique) solution $Q_\lambda$ satisfying $\int_{\R^d}Q_\lambda(x)^2\,\rd x=\lambda$ solves the same equation~\eqref{eq:NLS} with $\mu$ replaced by 
\begin{equation}
\mu(\lambda)=\mu\,\lambda^{\frac2d \frac{p-1}{1+\frac2d-p}}.
\label{eq:formula_mu_lambda}
\end{equation}
Under the additional condition
$$1<p<1+\frac{2}d,$$
it is useful to introduce the associated NLS functional 
\begin{equation}
 \cE(u):=\int_{\R^d}|\nabla u(x)|^2\,\rd x-\frac1p\int_{\R^d}|u(x)|^{2p}\,\rd x.
 \label{eq:NLS_E}
\end{equation}
Then $Q_\lambda$ is the unique solution (modulo phases and space translations) to the minimisation problem
\begin{equation}
\boxed{ I(d, p, \lambda) :=\min \left\{ \cE(u), \quad u\in H^1(\R^d), \ \int_{\R^d}|u|^2=\lambda \right\}. }
\label{eq:I_lambda}
\end{equation}
This explains the denomination {\em ground state}. When the values of $d$ and $p$ are clear from the context, we will omit them in our notation and write
$$I(\lambda) := I(d, p, \lambda).$$
When $1 < p < 1 + \frac2d$, the Gagliardo-Nirenberg inequality implies that $I(\lambda)$ is finite. For larger $p$'s one has to optimise a different functional~\cite{Weinstein-83}. 

By scaling one finds that 
\begin{equation}
 I(\lambda)=I(1)\,\lambda^{1+\frac{2}{d}\frac{p-1}{1+\frac2d -p}}.
 \label{eq:formula_I_lambda}
\end{equation}
Since the exponent is greater than $1$ with $I(1)<0$, this implies immediately that $\lambda\mapsto I(\lambda)$ is strictly concave over $\R_+$ and that 
\begin{equation}
 I(\lambda)<I(\lambda-\lambda')+I(\lambda'),\qquad \forall 0<\lambda'<\lambda.
 \label{eq:binding_NLS}
\end{equation}
These so-called \emph{binding inequalities} guarantee the existence of a minimiser and the compactness of all the minimising sequences up to translations, by the concentration-compactness method~\cite{Lions-82a,Lions-84,Lions-84b,Lewin-conc-comp}.

\subsection{The case of orthonormal functions: main results}
When studying fermions, like electrons, neutrons or protons, one is naturally led to deal with systems of orthonormal functions~\cite{LieSei-09}, that is, $u_1,...,u_N\in L^2(\R^d,\C)$ with $\pscal{u_j,u_k}=\delta_{jk}$. In this paper we study the nonlinear Schr\"odinger equation for such orthonormal systems, which could be also called the \emph{fermionic NLS equation}. It takes the form
\begin{equation}
\left(-\Delta-\bigg(\sum_{j=1}^N|u_j|^2\bigg)^{p-1}\right)u_i=\mu_i\,u_i,\qquad i=1,...,N.
\label{eq:fermionic_NLS}
\end{equation}
This is a system of $N$ coupled partial differential equations, where the coupling involves only the \emph{density}
$$\rho(x)=\sum_{j=1}^N|u_j(x)|^2$$
of the $N$ particles. Compared to~\eqref{eq:NLS} and~\eqref{eq:formula_mu_lambda}, we see that $\rho$ plays the role of $Q^2$, so $N = \int_\R \rho$ plays here the role of $\lambda$. In what follows, we use the notation $\lambda = N \in \N$ if it is integer valued.
Equation~\eqref{eq:fermionic_NLS} has already been studied in the mathematical~\cite{LinWei-05,AmbCol-06,BarWan-06,BarWanWei-07} and physical~\cite{Manakov-74,KevFra-16, ZhaYan-18} literature, but the constraint $\pscal{u_j,u_k}=\delta_{jk}$ is often not assumed. Dealing with this constraint is the main goal of our work. 

Equation~\eqref{eq:fermionic_NLS} has several invariances. The first is the invariance under the non-compact group of space translations and it corresponds to replacing all the functions $u_j(x)$ by $u_j(x-\tau)$ for some $\tau\in\R^d$. The second is due to rotations with $u_j(\cR x)$ and $\cR\in{\rm SO}(d)$. On the other hand, the multiplication by a phase for $N=1$ is here replaced by the action of the (compact) group $U(N)$ of space-independent $N\times N$ unitary matrices, in the manner
$$U\in U(N)\mapsto U\cdot(u_1,...,u_N)=\left(\sum_{k=1}^NU_{kj}u_k\right)_{j=1}^N.$$
This action does not affect the orthonormality constraint, nor the density $\rho$, hence it preserves the mean-field operator $-\Delta-\rho^{p-1}$. However it has the effect of transforming the diagonal matrix $\underline\mu={\rm diag}(\mu_1,...,\mu_N)$ of Lagrange multipliers into $U\underline\mu U^*$. Although it could seem more appropriate to start with a general hermitian matrix of multipliers $(\mu_{jk})_{1\leq j,k\leq N}$ associated with the constraints $\pscal{u_j,u_k}=\delta_{jk}$, we have for convenience chosen in~\eqref{eq:fermionic_NLS} a gauge in which this matrix is diagonal. 

Next, we introduce the concept of ground states for~\eqref{eq:fermionic_NLS}. We ask that the Lagrange multipliers $\mu_i$ are the $N$ first eigenvalues of the corresponding operator $-\Delta-\rho^{p-1}$. We always use the convention that the eigenvalues are ordered increasingly and are repeated in case of degeneracies. 

\begin{definition}[Ground state]\label{def:ground_state}
A \emph{ground state} is a system $(u_1,...,u_N)\in H^1(\R^d)$, orthonormal in $L^2(\R^d)$, which solves the equations~\eqref{eq:fermionic_NLS} where 
$$\mu_1<\mu_2\leq\cdots\leq\mu_N\leq0$$
are the \emph{$N$ first eigenvalues of the operator $-\Delta-\rho^{p-1}$}. 
\end{definition}

This definition coincides with the (strict) positivity of $Q$ in the case $N=1$, since the first eigenvalue $\mu_1$ is always non-degenerate with a positive eigenfunction, when it exists. Our definition for $N\geq2$ is further discussed in light of the $N$-particle problem in Remark~\ref{rmk:N-particle} below. 

It is interesting to determine the values of $p$ and $N$ (depending on the dimension $d\geq1$) for which ground states exist. In this article, we focus on the case $1 < p < 1 + \frac2d$, which can be recast into a minimisation problem. Some kind of ground states have recently been constructed in~\cite{HonKwoYoo-19} for $p>1+\frac2d$ but their corresponding density $\rho\in L^p(\R^d)$ is not necessarily in $L^1(\R^d)$, so $N = \int_{\R^d} \rho$ may in fact be infinite. The critical case $p=1+\frac2d$ will be handled in the second part~\cite{FraGonLew-20_ppt} of this work, in dimensions $d\geq3$.

A more difficult question is that of the uniqueness of ground states, when they exist. We believe that in dimension $d = 1$, ground states are always unique up to translations whenever they exist. Numerical simulations in dimension $d=2$ presented later in Figure~\ref{fig:2d} suggest that the system may break rotational symmetry, in which case minimisers are not unique modulo space translations. 

From now on, we assume
$$\boxed{1<p<1+\frac{2}{d}.}$$
As in the $N = 1$ case, ground states naturally occur as minimisers of the associated nonlinear functional
\begin{equation}
\cE(u_1,...,u_N)=\sum_{i=1}^N\int_{\R^d}|\nabla u_i(x)|^2\,\rd x-\frac1p\int_{\R^d}\left(\sum_{i=1}^N|u_i(x)|^2\right)^p\,\rd x.
\label{eq:def_energy_u_j}
\end{equation}
This paper is devoted to the study of the associated minimisation problem
\begin{equation}
\boxed{J(N)=\inf \Big\{ \cE(u_1,...,u_N), \quad u_1,...,u_N\in H^1(\R^d,\C), \  \pscal{u_j,u_k}_{L^2}=\delta_{jk}\Big\}.}
\label{eq:def_J_integers}
\end{equation}
For $N = 1$, we recover $J(1)=I(1)$. 
Unfortunately, there is no simple formula such as~\eqref{eq:formula_I_lambda} for $J(N)$. This is because of the orthonormality constraint, which prevents us from multiplying $u_i$ by a positive constant, as one does for $I(N)$ to obtain~\eqref{eq:formula_I_lambda}. 

The goal of the present article is to prove that $J(N)$ admits minimisers, for some values of $p$ and $N$. Following~\cite{Lewin-11} and as is usual in the study of nonlinear elliptic minimisation problems, our main strategy is to prove the so-called {\em binding inequalities}. Indeed, as we recall in Theorem~\ref{th:Lewin} below, if $N \in \N$ is such that the following binding inequalities hold:
\begin{equation} \label{eq:condition_binding_conjecture}
     J(N) < J(N-K)+J(K),\qquad \text{for all}\quad K=1,...,N-1,
\end{equation}
then $J(N)$ has a minimiser and all the minimising sequences are compact, up to translations.  We prove later in Proposition~\ref{lem:Euler-Lagrange} that minimisers of $J(N)$ are indeed ground states. Therefore, the problem boils down to proving the binding inequalities~\eqref{eq:condition_binding_conjecture}. We believe that the following holds.

\begin{conjecture}[Binding]\label{conj:binding}
For every $N\geq2$ and every
\begin{equation}
1<p<\min\left(2,1+\frac{2}{d}\right),
\label{eq:cond_p_conjecture}
\end{equation}
the binding inequalities~\eqref{eq:condition_binding_conjecture} hold.
In particular, $J(N)$ admits a minimiser, which is a ground state for~\eqref{eq:fermionic_NLS}.
\end{conjecture}

Let us emphasise the new condition $p<2$. The critical exponent $p = 2$ appears naturally in our proof, when we evaluate the interaction between two ground states placed far away. But we will also explain in Theorem~\ref{thm:p=2} below that binding does \emph{not} occur at $p=2$ in dimension $d=1$, so that the condition~\eqref{eq:cond_p_conjecture} is in fact optimal. 

We were not able to prove Conjecture~\ref{conj:binding} in the whole range of parameters. We prove in this paper two weaker results. The first one is that the conjecture holds for $p$ close enough to~$1$.

\begin{theorem}[Binding and existence of ground states for small $p$]\label{thm:ground_states_orthonormal_small_p}
For all $d \ge 1$, there exists $1<p_c(d) \le 1+\frac2d$ such that, for all 
$$1 < p < p_c(d),$$ 
the binding inequalities~\eqref{eq:condition_binding_conjecture} hold for all $N\geq2$. In particular, for all $1 < p < p_c(d)$ and all $N \ge 2$ there exists a minimiser $(u_1,...,u_N)$ for $J(N)$, which is a ground state. It solves the nonlinear system~\eqref{eq:fermionic_NLS} where the corresponding multipliers satisfy
$$\mu_1<\mu_2\leq\mu_3\leq\cdots\leq\mu_N<0$$ 
and are the $N$ first eigenvalues of the Schr\"odinger operator $-\Delta-(\sum_{i=1}^N|u_i|^2)^{p-1}$, counted with multiplicity.
\end{theorem}

In the proof in Section~\ref{sec:proof_existence_small_p} we give an explicit lower bound of the critical exponent $p_c(d)$. This lower bound depends solely on the NLS solution $Q$ of Equation~\eqref{eq:NLS}. It can be numerically computed with very high accuracy using Runge-Kutta numerical methods, since $Q$ is radial hence solves a second order Ordinary Differential Equation. We find
\begin{equation}
p_c(d)>\begin{cases}
  1.614& \text{for $d=1$,}\\
  1.530& \text{for $d=2$,}\\
  1.444& \text{for $d=3$.}\\
  \end{cases} 
 \label{eq:values_p_c_d}
\end{equation}
In particular, we are able to cover the important case $p=4/3$ in dimension $d=3$, which is the object of Section~\ref{sec:rHFD}. These bounds show that the critical $p_c(d)$ is not so close to 1 and let us think that $p_c(d)=\min(2,1+\frac2d)$ should hold. See Remark~\ref{rmk:comput_p_c} below for more comments on $p_c(d)$.

Our second result can cover the whole range $1 < p < \min (2, 1 + \frac2d)$, but is valid only for an infinite sequence $N_j\to\ii$, including the two-particle case $N=2$. 

\begin{theorem}[Binding and existence of ground states for all $p$]\label{thm:ground_states_orthonormal_all_p}
Let $d\geq1$ and
$$1<p<\min\left(2\;,\; 1+\frac{2}{d}\right).$$
There exists an infinite increasing sequence of integers 
$$N_1=1,\ N_2=2< N_3<\cdots < N_j < \cdots $$ 
for which the binding inequalities~\eqref{eq:condition_binding_conjecture} hold. For any such $N=N_j$, $J(N)$ has a minimiser $(u_1,...,u_N)$ which satisfies the same properties as in Theorem~\ref{thm:ground_states_orthonormal_small_p}. 
\end{theorem}

In Section~\ref{sec:rHFD} we use Theorem~\ref{thm:ground_states_orthonormal_all_p} to prove translational symmetry breaking of the Kohn-Sham model for a crystal with a large Dirac exchange coefficient. This result is in the same spirit as the recent work~\cite{Ricaud-17} by Ricaud on the Thomas-Fermi-von Weizs\"acker-Dirac model, and it was indeed our first motivation for studying the fermionic NLS equation~\eqref{eq:fermionic_NLS}. This problem naturally brings the case $d=3$ and $p=4/3$ which is covered by Theorem~\ref{thm:ground_states_orthonormal_all_p}. 

When $p=1+2/d$ the system has an additional invariance and it is not appropriate to fix the constant in front of the nonlinear term $\int_{\R^d}(\sum_{i=1}^N|u_i|^2)^p$ in~\eqref{eq:def_energy_u_j} to be $1/p$. If we study the minimisation problem similar  to~\eqref{eq:def_J_integers} with a constant $\alpha>0$ in front of the nonlinear term, we obtain that there exists a critical $\alpha_c(N)>0$ such that $J(N)=0$ for $\alpha\leq \alpha_c(N)$ and $J(N)=-\ii$ for $\alpha>\alpha_c(N)$. There are no minimisers for $0\leq\alpha<\alpha_c(N)$. In the second part~\cite{FraGonLew-20_ppt} of this work, R.L.~Frank and the first two authors prove a result similar to Theorem~\ref{thm:ground_states_orthonormal_all_p} for $\alpha=\alpha_c(N_j)$ with $N_j\to\ii$, in dimensions $d\geq3$. 

As we already mentioned, the threshold $p=2$ appears naturally in our proof (Proposition~\ref{lem:strongBinding}). So our results do not cover the case $2 \le p < 3$ in dimension $d = 1$. Actually, binding probably never holds for $p \in [2, 3)$ in dimension $d=1$.  

\begin{conjecture}[Absence of binding in 1D for $p\in[2,3)$]\label{conj:1D}
    In dimension $d=1$, for all $2 \leq p < 3$ and all $N \ge 2$, $J(N)$ does not have minimiser, and we have
    $$J(N)=NJ(1)=NI(1).$$
\end{conjecture}

It is explained in~\cite[Remark 14]{FraGonLew-20_ppt} that Conjecture~\ref{conj:1D} follows from the one-dimensional Lieb-Thirring conjecture in~\cite{LieThi-76}. Without entering into the details and using the notation of~\cite{LieThi-76,Frank-20_ppt}, this conjecture states that the best Lieb-Thirring constant $L_{\gamma, 1} $ in dimension $d=1$ coincides with the one-bound-state (Gagliardo-Nirenberg) constant $L_{\gamma, 1}^{(1)}$ for all $1/2 \le \gamma \le 3/2$. In fact, Lieb and Thirring proved in~\cite{LieThi-76} that $L_{3/2, 1} = L_{3/2,1}^{(1)}$ for $\gamma = 3/2$ and this implies the following result, proved in~\cite{FraGonLew-20_ppt}.

\begin{theorem}[Non-existence for $d=1$, $p=2$~\cite{FraGonLew-20_ppt}]\label{thm:p=2}
Let $d=1$ and $p=2$. Then we have $J(N)=N\,J(1)=N\,I(1)$  for all $N\in\N$. In addition, $J(N)$ admits no minimiser for $N\geq2$.
\end{theorem}

\subsection{Comments and ideas of proof}
This section contains additional comments about Theorems~\ref{thm:ground_states_orthonormal_small_p} and~\ref{thm:ground_states_orthonormal_all_p} as well as some proof ideas. 

First, we place our result in a more general context. Several tools of nonlinear analysis have been generalised to systems of orthonormal functions, which can also be seen as random fields~\cite{Suzzoni-15_ppt,ColSuz-19_ppt}. The most celebrated example is the \emph{Lieb-Thirring inequality}~\cite{LieThi-75,LieThi-76,LieSei-09,FraHunJexNam-18_ppt} which states that 
\begin{equation}
\sum_{n=1}^N\int_{\R^d}|\nabla u_n(x)|^2\,\rd x\geq c_{\rm LT}(d)\int_{\R^d}\left(\sum_{n=1}^N|u_n(x)|^2\right)^{1+\frac2d}\,\rd x
\label{eq:LT_orthonormal}
\end{equation}
where the positive constant $c_{\rm LT}(d)>0$ is independent of $N$. This important inequality replaces the Gagliardo-Nirenberg inequality for large orthonormal systems. It will play a role in our analysis of the large-$N$ behaviour of $J(N)$ later in Section~\ref{sec:large_N}. Other Gagliardo-Nirenberg-type inequalities were considered in~\cite{Lieb-83d}. More recently, the Strichartz inequality has been extended to orthonormal systems in~\cite{FraLewLieSei-14,FraSab-17,BezHonLeeNakSaw-17} and it has played a central role for the existence and the long time behaviour of infinite systems~\cite{LewSab-15a,LewSab-14b,CheHonPav-17,CheHonPav-17b,ColSuz-19_ppt}. The fermionic NLS time-dependent equation with $N=+\ii$ has been studied in~\cite{CheHonPav-17b}. 

After all these works on systems of orthonormal functions, investigating ground states of the fermionic NLS equation~\eqref{eq:fermionic_NLS} seems a natural next step. There are many open questions and we hope that our paper will stimulate more work on the problem. 

\medskip

Next, we briefly explain our strategy of proof for Theorems~\ref{thm:ground_states_orthonormal_small_p} and~\ref{thm:ground_states_orthonormal_all_p} and discuss the main message stemming from this analysis. Assuming that $J(N)$ and $J(M)$ have minimisers, the proof of the binding inequality $J(N+M)<J(N)+J(M)$ goes by evaluating the nonlinear interaction of these two minimisers placed far away, as is classical in such variational problems. The main difficulty here is that the functions $u_i$ all decay exponentially fast at infinity so that this interaction is exponentially small. Our main result follows from a careful evaluation of this term. In informal words, our main message is thus that \textbf{quantum tunnelling can induce binding (hence existence of ground states) in a nonlinear model, provided that the nonlinearity is `sufficiently strong'}. Here `strong' is precisely the condition that $p<2$. Our simulations presented below in Section~\ref{sec:large_N} and in~\cite{FraGonLew-20b_ppt} suggest that this is in fact a \textbf{real physical effect} in dimensions $d=1,2$, not just a mathematical argument. Namely, we have numerically found ground states for $J(N)$ which are very close to being a combination of $N$ copies of the NLS solution $Q$ (see Figure~\ref{fig:2d}), with an energy very close to $N\,J(1)$. 

To our knowledge, this is the first result of this type for a nonlinear translation-invariant problem. A similar evaluation of exponentially small nonlinear interactions can be for instance found in~\cite{Albanese-88,CatLio-93a,CatLio-93b,CatLio-93c,BahLi-90,BahLio-97,Daumer-94,OlgRou-20_ppt} but for completely different models which are not translation-invariant and have no orthonormal constraint. In~\cite{FraGonLew-20_ppt} our strategy is used to disprove part of the Lieb-Thirring conjecture~\cite{LieThi-76}. 

Let us now explain where the condition $p<2$ arises in the proof, which can be found in Section~\ref{sec:proof_binding}. We assume that there are two orthonormal systems $(u_n)_{n=1}^N$ and $(v_m)_{m=1}^M$ which are minimisers for $J(N)$ and $J(M)$, respectively. We denote the two densities by $\rho=\sum_{n=1}^N|u_n|^2$ and $\rho'=\sum_{m=1}^M|v_m|^2$, the decay of which is dictated by the last eigenfunctions $u_N$ and $v_M$ which have the smallest eigenvalues $\mu_N$ and $\mu'_M$ in absolute value. Next we add to the $u_n$'s a translation of the $v_m$'s by a large amount $R$ in the direction of $e_1=(1,0,...,0)$. The resulting family of $N+M$ functions is not orthonormal but it is exponentially close to being so. Orthonormalising it generates an error of the order $O(e_R^2)$ in the energy, where
$$e_R:=\max_{n,m}\int_{\R^d}|u_n(x)|\,|v_m(x-Re_1)|\rd x$$
is the largest possible overlap between these functions. In the limit $R\to\ii$, $e_R$ essentially behaves like 
$\exp(-\min(|\mu_N|^{\frac12},|\mu'_M|^{\frac12})R)$ due to the exponential decay of the eigenfunctions. On the other hand, the nonlinear term is always attractive and a careful evaluation in the intermediate region gives that it is at least of the order
$$\exp\left(-\frac{\sqrt{|\mu_N|\,|\mu'_M|}}{\sqrt{|\mu_N|}+\sqrt{|\mu'_M|}}p R\right).$$
Therefore, the nonlinearity wins over the orthonormalisation error under the assumption that 
\begin{equation}
    1 < p < 1 + \sqrt{ \frac{\min(| \mu_N |, | \mu'_M | ) }{ \max ( | \mu_N|, | \mu'_M |) }  }
\label{eq:cond_p_mu_mu}
\end{equation}
where the right side is always less or equal than $2$. This is how the condition $p<2$ occurs. The difficulty with~\eqref{eq:cond_p_mu_mu} is that the Lagrange multipliers $\mu_N$ and $\mu'_M$ are unknown. To prove Theorem~\ref{thm:ground_states_orthonormal_small_p}, we derive universal lower and upper bounds (independent of $N$) on the last eigenvalue $\mu_N$. This gives a critical exponent $p_c(d)$ below which binding holds. To prove Theorem~\ref{thm:ground_states_orthonormal_all_p}, we observe that if $N = M$, then we can choose the same minimiser for $J(N)$ and $J(M)$ so that $\mu_N = \mu_M$. In this simpler case the interaction is attractive whenever $1<p<\min(2,1+2/d)$. So if $J(N)$ has a minimiser, then
\begin{equation}
J(2N)<2J(N).
\label{eq:binding_two_copies}
\end{equation}
Since there are ground states for $N=1$, we deduce that there are ground states for $N=2$. By a simple pigeon-hole principle, we are then able to deduce that binding holds for an infinite sequence $N_j\to\ii$, using only~\eqref{eq:binding_two_copies}. 

\medskip

We end this section with two additional comments about Theorem~\ref{thm:ground_states_orthonormal_small_p} and~\ref{thm:ground_states_orthonormal_all_p}. First, we link our NLS equation for fermionic systems to a kind of Gagliardo-Niremberg-Sobolev inequality for orthonormal functions.

\begin{remark}[Gagliardo-Niremberg-Sobolev for orthonormal systems]\label{rmk:rescale}\rm 
 If we rescale all the $u_n$ in the manner $\alpha^{d/2}u_n(\alpha x)$ and optimise over $\alpha$ we obtain the inequality
\begin{equation}
  N^{\frac{2}{d(p-1)}-1}\sum_{n=1}^N\int_{\R^d}|\nabla u_n(x)|^2\,\rd x\geq c(d,p,N)\; \left(\int_{\R^d}\bigg(\sum_{n=1}^N|u_n(x)|^2\bigg)^p\,\rd x\right)^{\frac{2}{d(p-1)}}
 \label{eq:other_inequality}
\end{equation}
with the best constant
\begin{equation}
 c(d,p,N)=\left(\frac{N}{-J(N)}\right)^{\frac{1+\frac2d-p}{p-1}}\left(\frac{d}{2p}\right)^{\frac{2}{d(p-1)}}(p-1)\left(1+\frac2d-p\right)^{\frac{1+\frac2d-p}{p-1}}.
\end{equation}
The constant $c(d,p,N)$ has a finite limit when $N\to\ii$, as we will prove in the next section.
Our theorems give the existence of optimisers for this inequality (either for small $p$ or for a subsequence $N_j\to\ii$). The (non-sharp) inequality~\eqref{eq:other_inequality} easily follows from the Lieb-Thirring inequality, using $\int_{\R^d}\sum_{n=1}^N|u_n|^2=N$ together with H\"older's inequality. In the critical case $p = 1 + \frac{2}{d}$, the two terms in the energy scale similarly, so one cannot deduce an equality similar to~\eqref{eq:other_inequality}. See~\cite{FraGonLew-20_ppt} for more about~\eqref{eq:other_inequality} and its link with Lieb-Thirring inequalities.\qed
\end{remark}

Our next remark is about how to interpret our result within the framework of $N$-particle anti-symmetric wave functions. The idea is that without a two-body interaction term and with a concave nonlinearity depending only on the density, a minimisation problem set on anti-symmetric wave functions can be restricted to Slater determinants.

\begin{remark}[Interpretation in terms of $N$-particles]\label{rmk:N-particle}\rm 
For a wave function $\Psi \in L^2((\R^{d})^N, \C)$ with $\| \Psi \|_{L^2((\R^{d})^N, \C)} = 1$, consider the energy functional
\begin{equation}
   \cE_{\rm QM}(\Psi) := \int_{(\R^d)^N}|\nabla\Psi(x_1,...,x_N)|^2\,\rd x_1\cdots \rd x_N-\frac1p\int_{\R^d}\rho_\Psi(x)^{p}\,\rd x
   \label{eq:NLS_many_body}
\end{equation}
where the density $\rho_\Psi$ is defined by
\begin{multline*}
\rho_\Psi(x)=\int_{(\R^d)^{N-1}}|\Psi(x,x_2,...,x_N)|^2\,\rd x_2 \cdots \rd x_N+\cdots \\\cdots+\int_{(\R^d)^{N-1}}|\Psi(x_1,x_2,...,x)|^2\,\rd x_1 \cdots \rd x_{N-1}. 
\end{multline*}
Minimisation problems involving functionals of the type~\eqref{eq:NLS_many_body} (posed on the $N$-particle space with a nonlinear term depending on $\rho_\Psi$) have been studied in~\cite{Lewin-11}. Here, without further constraints on $\Psi$, the minimum of $\cE_{\rm QM}$ on the unit sphere is attained for a symmetric (that is, bosonic) wave function, which forms a Bose-Einstein condensate on the NLS ground state $Q_N/\sqrt{N}$ defined in~\eqref{eq:I_lambda}:
\[
    \Psi(x_1, \cdots, x_N) = \frac{\re^{\ri \theta}}{N^{\frac{N}2}} Q_{N}(x_1) \cdots Q_N(x_n).
\]
This is a simple consequence of the Hoffmann-Ostenhof inequality~\cite{Hof-77}
$$\int_{(\R^d)^N}|\nabla\Psi(x_1,...,x_N)|^2\,\rd x_1\cdots \rd x_N\geq\int_{\R^d}|\nabla\sqrt{\rho_\Psi}(x)|^2\,\rd x$$
which implies that 
$$\cE_{\rm QM}(\Psi)\geq \cE(\sqrt{\rho_\Psi})$$
where we recall that $\cE$ is the NLS energy~\eqref{eq:NLS_E}. In other words, the unconstrained $N$-particle problem is the same as the NLS problem for one function~\eqref{eq:I_lambda}. Note that when $N\to\ii$ the system collapses since $Q_N$ is a rescaling of $Q$ by the factor $N^{-(p-1)/(d+2-dp)}$. 

The situation is different if we restrict the minimisation to anti-symmetric (that is, fermionic) wave functions. From the arguments in~\cite{Lewin-11} and in Lemma~\ref{lem:gamma_wants_to_have_finite_rank} below, it follows that minimisers are Slater determinants (also called Hartree-Fock states), that is, of the form
$$\Psi(x_1,...,x_N)=\frac1{\sqrt{N!}}\det(u_j(x_k))_{1\leq j,k\leq N}$$
where $u_1,...,u_N$ form an orthonormal system in $L^2(\R^d)$. Slater determinants are the least correlated wave functions compatible with the anti-symmetric constraint. These wave functions satisfy that $\cE_{\rm QM}(\Psi)=\cE(u_1,...,u_N)$, our NLS functional in~\eqref{eq:def_energy_u_j}. The $N$-particle interpretation of $J(N)$ in~\eqref{eq:def_J_integers} is therefore that it corresponds to minimising $\cE_{\rm QM}$ over anti-symmetric wave functions. In this light our Definition~\ref{def:ground_state} of a `ground state' is justified since the corresponding $N$-particle wave function $\Psi$ is indeed a minimiser of $\cE_{\rm QM}$, in the anti-symmetric subspace. On the contrary to the bosonic case, the fermionic model is $H$-stable~\cite{Ruelle} since $J(N)$ behaves linearly in $N$ in the limit $N\to\ii$, as discussed in the next section.\qed
\end{remark}

\subsection{The large--$N$ limit}\label{sec:large_N}
Our main results, Theorems~\ref{thm:ground_states_orthonormal_small_p} and~\ref{thm:ground_states_orthonormal_all_p}, imply that $J(N)$ admits a minimiser, for all $N$ or for a subsequence. It then seems natural to ask what is happening in the limit $N\to\ii$. This section contains results and comments in this direction. 

Several possible scenarios come to mind. A very natural possibility is that a sequence of minimisers $\rho_N=\sum_{n=1}^N|u_n|^2$ would converge (e.g. in $L^\ii_{\rm loc}(\R^d)$ after an appropriate translation) to some limit $\rho_\ii$ which is extended over the whole space. This density $\rho_\ii$ could for instance be constant (semi-classical or fluid phase), or a non-trivial periodic function (crystallisation, or solid phase~\cite{BlaLew-15}). We think that these are the only two possibilities and we explain in this section which one we expect depending on the values of $p$ and $d$. We refer to~\cite{FraGonLew-20_ppt,FraGonLew-20b_ppt} for a similar discussion in the context of Lieb-Thirring inequality.

\subsubsection{The NLS Thomas-Fermi problem}
We first introduce the NLS Thomas-Fermi problem~\cite{Lieb-83b}, which will give the value of $\rho_\ii$ in case it is constant over the whole space. For $d\geq1$ and $1<p<1+\frac{2}{d}$ we introduce the minimisation problem
\begin{equation}
 \min_{\substack{\rho\geq0\\ \int_{\R^d}\rho=N}}\int_{\R^d}\left(C\rho(x)^{1+\frac2d} - \frac1p\rho(x)^p\right)\,\rd x
 \label{eq:Thomas-Fermi}
\end{equation}
where $C>0$ is a constant to be chosen later.
This problem can be solved explicitly. The next statement will be used several times in the paper and it states that the optimisers are for all $N$ exactly equal to some constant $\rho_*$ depending only on $C$, $d$ and $p$, on a set $\Omega$ of measure $|\Omega|=N/\rho_*$.

\begin{lemma}[Thomas-Fermi has constant optimisers] \label{lem:LTminimisation}
    Let $d\geq1$, $1<p<1+\frac2d$ and $C>0$. We have
\[ 
   \min_{\substack{\rho\geq0\\ \int_{\R^d}\rho=N}}\int_{\R^d}\left(C \rho(x)^{1+\frac2d}-\frac1p\rho(x)^p\right)\,\rd x = -N\left(1+\frac2d-p\right)\frac{d}{2p}\left(\frac{d(p-1)}{2pC}\right)^{\frac{p-1}{1+\frac2d-p}}
\]
with equality if and only if $\rho(x)=\rho_*\1_\Omega(x)$ for some Borel set $\Omega$ of measure $|\Omega|=N/\rho_*$, where
\begin{equation}
 \rho_*=\underset{\rho>0}{\rm argmin}\left(C \rho^{\frac2d}-\frac1p\rho^{p-1}\right)=\left(\frac{d(p-1)}{2pC}\right)^{\frac1{1+\frac2d-p}}.
 \label{eq:formula_rho_*}
\end{equation}
\end{lemma}

\begin{proof}
For $\alpha_*:=C\rho_*^{\frac2d}-\frac1p\rho_*^{p-1}$, the map
$$\rho\in\R_+\mapsto C \rho^{1+\frac2d}-\frac1p\rho^p-\alpha_*\rho$$
is non-negative over $\R_+$ and admits exactly the two zeros $0$ and $\rho_*$. Thus we have 
\begin{multline*}
\int_{\R^d}\left(C \rho(x)^{1+\frac2d}-\frac1p\rho(x)^p\right)\,\rd x-\alpha_*N\\=\int_{\R^d}\left(C \rho(x)^{1+\frac2d}-\frac1p\rho(x)^p-\alpha_*\rho(x)\right)\,\rd x\geq0 
\end{multline*}
with equality if and only if $\rho$ takes only the two values $0$ and $\rho_*$, hence is of the form $\rho(x) = \rho_* \1_{\Omega}(x)$, where $\Omega$ is a Borel set with $| \Omega |  = N/\rho_*$. The minimum in the statement is thus equal to $\alpha_*N$ and its value follows after a computation. 
\end{proof}

By Lemma~\ref{lem:LTminimisation} we see that minimisers of the Thomas-Fermi problem~\eqref{eq:Thomas-Fermi} are not unique at all, even up to translations, since any Borel set $\Omega$ is allowed as soon as $|\Omega|=N/\rho_*$. In addition, even if minimisers exist, the binding inequalities fail in NLS Thomas-Fermi theory and the energy is perfectly additive. 

\subsubsection{Limit $N\to\ii$ and link with the Lieb-Thirring conjecture}
Next, we prove that $J(N)/N$ admits a limit when $N\to\ii$ and derive upper and lower bounds on its value, in terms of the Thomas-Fermi problem~\eqref{eq:Thomas-Fermi} for two possible $C$'s. Our main result in this section is the following

\begin{theorem}[Large$-N$ limit]\label{thm:limit_N}
Let $d\geq1$ and $1<p<1+{2}/{d}$. The limit
\begin{equation}
 \boxed{e(d,p):=\lim_{N\to\ii}\frac{J(N)}{N}=\inf_{N\geq1}\frac{J(N)}{N}}
 \label{eq:limit_N_infty}
\end{equation}
exists and satisfies
\begin{equation}
 e_{\rm LT} (d,p) \le e(d, p) \le \min\big\{e_{\rm sc} (d,p),I(1)\big\} < 0
 \label{eq:upper_lower_bd_e}
\end{equation}
where
\begin{equation}
e_{\rm LT}(d, p) :=  - \left(1+\frac2d-p\right)\frac{d}{2p}\left(\frac{d(p-1)}{2p\,c_{\rm LT}(d)}\right)^{\frac{p-1}{1+\frac2d-p}}
\label{eq:e_LT}
\end{equation}
and 
\begin{equation}
e_{\rm sc}(d, p) :=  - \left(1+\frac2d-p\right)\frac{d}{2p}\left(\frac{d(p-1)}{2p\,c_{\rm sc}(d)}\right)^{\frac{p-1}{1+\frac2d-p}}
\label{eq:e_TF}
\end{equation}
are the Thomas-Fermi energies obtained for $C$ respectively equal to the Lieb-Thirring constant $c_{\rm LT}(d)$ in~\eqref{eq:LT_orthonormal} and to the semi-classical constant
\begin{equation}
 c_{\rm sc}(d)=\frac{ 4\pi^2 d}{(d+2)} \left( \dfrac{d}{| \mathbb{S}^{d-1} |} \right)^{\frac2d}.
 \label{eq:c_TF}
\end{equation}
\end{theorem}

This theorem is a corollary of our other results in this paper. Its detailed proof is provided later in Section~\ref{sec:proof_large-N-limit}. The lower bound in~\eqref{eq:upper_lower_bd_e} is an immediate consequence of the Lieb-Thirring inequality~\eqref{eq:LT_orthonormal}. The existence of the limit~\eqref{eq:limit_N_infty} follows from the fact that $J(N)$ is subadditive. To understand the meaning of the semi-classical number $e_{\rm sc}(d,p)$, take for instance the $N$ first eigenfunctions $u_i$ of the Dirichlet Laplacian in a large domain $\Omega$ of volume $|\Omega|=N/\rho_\ii$, for some constant $\rho_\ii$ to be determined. Plugging them in the energy $\cE$, we obtain an upper bound on $J(N)$. By semi-classical analysis, the corresponding density $\sum_{i=1}^N|u_i|^2$ is almost equal to the constant $\rho_\ii$ inside the domain $\Omega$. On the other hand the total kinetic energy of the $u_i$ is approximately given by Weyl's formula $c_{\rm sc}(d)\rho_\ii^{1+\frac2d} |\Omega|$ with the semi-classical constant in~\eqref{eq:c_TF}. Thus, to leading order the energy behaves as
$$\left(c_{\rm sc}(d)\rho_\ii^{1+\frac2d} -\frac{\rho_\ii^p}{p}\right)|\Omega|=\left(c_{\rm sc}(d)\rho_\ii^{\frac2d} -\frac{\rho_\ii^{p-1}}{p}\right)N.$$
From the proof of Lemma~\ref{lem:LTminimisation}, this is minimized for $\rho_\ii=\rho_*$ in~\eqref{eq:formula_rho_*} with $C=c_{\rm sc}(d)$. In other words, whenever $\rho_N$ converges to a constant, we expect it to be this $\rho_*$ and the limit $e(d,p)$ to be equal to $e_{\rm sc}(d,p)$. 

Note that the behaviour in $N$ found in~\eqref{eq:limit_N_infty} is very different from the `bosonic' case recalled above in~\eqref{eq:formula_I_lambda} and Remark~\ref{rmk:N-particle}. In the later case, $I(N)$ behaves super-linearly in $N$ and the exponent depends on $p$ and $d$. On the contrary, for systems of orthonormal functions (fermions), $J(N)$ behaves linearly for all admissible $p$ and~$d$.

The Lieb-Thirring conjecture~\cite{LieThi-76,Frank-20_ppt} states that in dimensions $d\geq3$, one has $c_{\rm LT}(d) = c_{\rm sc}(d)$. Should this conjecture be true, the upper and lower bounds would coincide in~\eqref{eq:upper_lower_bd_e} and we would thus deduce that 
\begin{equation}
e(d,p)=e_{\rm sc}(d,p)\qquad\text{for all $1<p<1+\frac2d$ and $d\geq3$.}
\label{eq:if_LT_holds}
\end{equation}
We therefore expect that \textbf{the density $\rho_N$ of a ground state for $J(N)$ should converge to the constant $\rho_*$ given by Lemma~\ref{lem:LTminimisation} with $C=c_{\rm sc}(d)$, for all $1<p<1+2/d$ in all dimensions $d\geq3$}. 

On the contrary, in dimensions $d=1,2$ it is known that $c_{\rm LT}(d)< c_{\rm sc}(d)$~\cite{LieThi-76,Frank-20_ppt} and the two bounds in~\eqref{eq:upper_lower_bd_e} do not coincide. We expect \textbf{the limiting density to be a non-trivial periodic function for all $1<p<2$ in dimensions $d=1,2$}. By Theorem~\ref{thm:p=2}, in dimension $d=1$ we think that the period will depend on $p$ and increase when $p\to2^-$, whereas the density converges to $Q^2$ in each unit cell. In dimension $d=1$, the periodicity of minimisers is confirmed by a numerical simulation\footnote{The code is available upon request to the authors.} reported on in Figure~\ref{fig:crystal_1D}. Showing such a fact is an interesting open problem~\cite{BlaLew-15}.  In dimension $d = 2$, we could not run the computations for a too large value of $N$ but the numerical results for $N\leq7$ presented in Figure~\ref{fig:2d} seem to suggest that the particles crystallise on a triangular lattice, as is often the case in two dimensions~\cite{BlaLew-15}. 

\begin{figure}[h!]
\includegraphics[width=10cm]{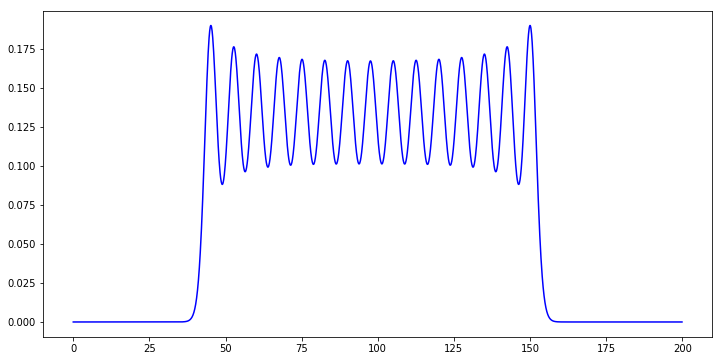}
\caption{Numerical computation of the density $\rho_N=\sum_{i=1}^{N}|u_i|^2$ of the minimiser for $J(N)$ in dimension $d=1$ with $N = 15$ and $p=1.3$. The system exhibits a crystallised phase with $15$ local maxima.\label{fig:crystal_1D}}
\end{figure}

\begin{figure}
    \centering
       
\includegraphics[width=4.2cm]{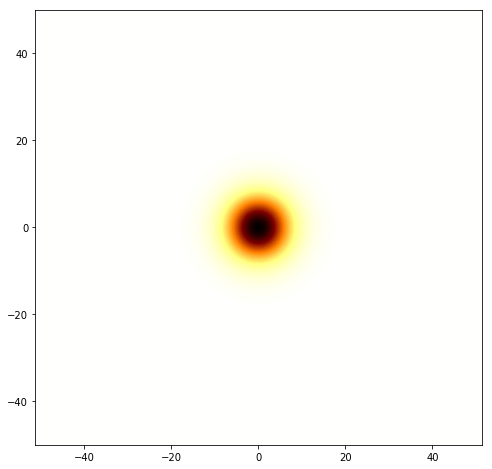}\hfill \includegraphics[width=4.2cm]{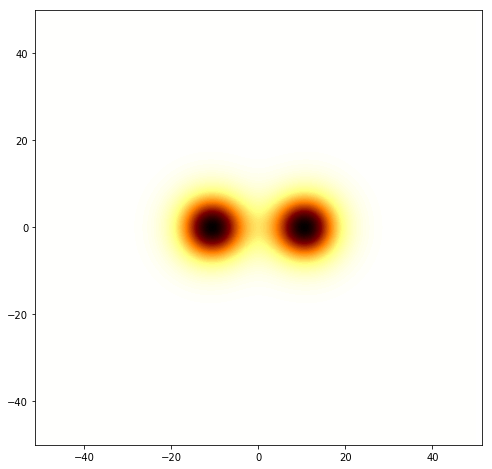}\hfill \includegraphics[width=4.2cm]{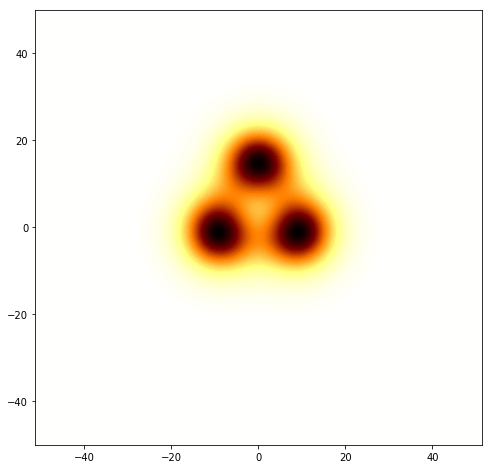}

\includegraphics[width=4.2cm]{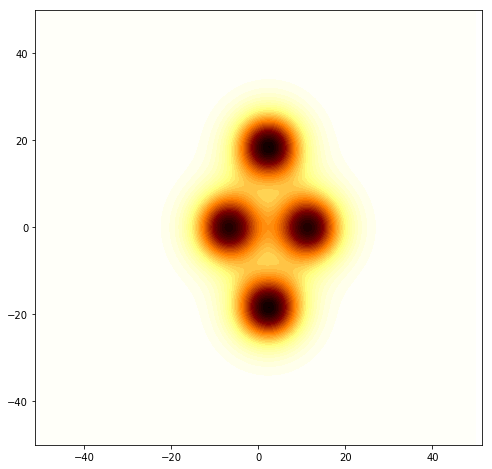}\hfill \includegraphics[width=4.2cm]{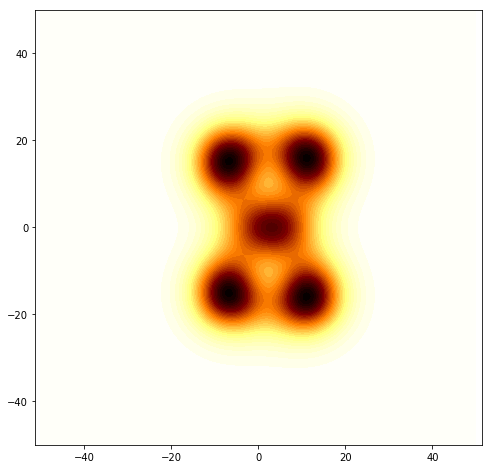}\hfill \includegraphics[width=4.2cm]{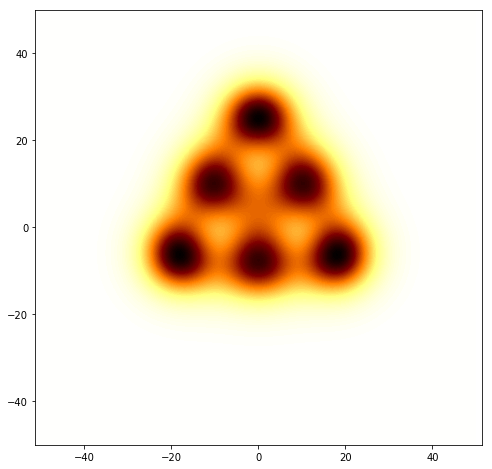}
        
\hfill \includegraphics[width=4.2cm]{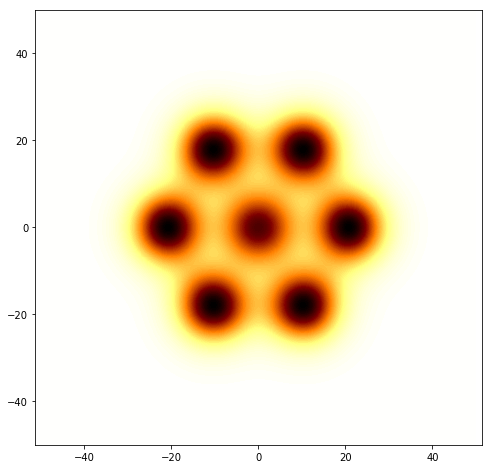}\hfill  \ 

    \caption{Numerical computation of the density $\rho_N=\sum_{i=1}^N|u_i(x)|^2$ of the minimiser of $J(N)$ for $N$ varying between $1$ and $7$ in dimension $d = 2$ with $p = 1.5$. \label{fig:2d}}
 
 \end{figure}
 
 \begin{figure}
  \centering
  \includegraphics[width=7cm]{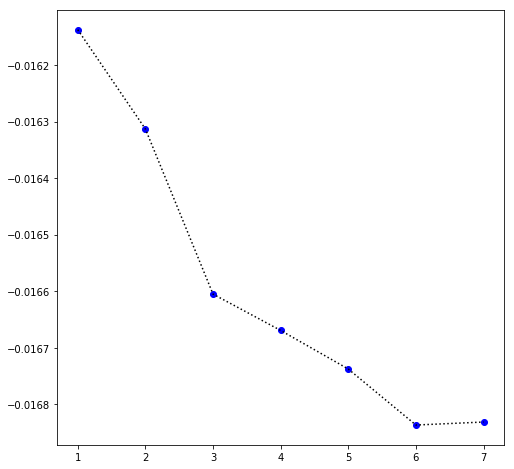}

    \caption{Plot of $J(N)/N$ as a function of $N$ in dimension $d=2$ for $p=1.5$. It is unclear whether $J(7)/7$ is smaller than $J(6)/6$, due to numerical errors.\label{fig:2d_bis}}
 \end{figure}

\section{Proof of Theorems~\ref{thm:ground_states_orthonormal_small_p}, \ref{thm:ground_states_orthonormal_all_p} and~\ref{thm:limit_N}}\label{sec:proof_thm}

\subsection{Relaxation via density matrices}

Here we introduce the relaxation of $J(N)$ using density matrices, a classical tool in the context of variational problems involving orthonormal functions~\cite{Lieb-81,Solovej-91,Bach-92,Bach-93,BacLieLosSol-94,BacLieSol-94,FraLieSeiSie-07,Lewin-11}. Let $\gamma=\gamma^*\geq0$ be a non-negative self-adjoint operator on $L^2(\R^d)$, with $\tr(\gamma)=\lambda>0$. Then $\gamma$ is compact and, by the spectral theorem, it can be diagonalised in the form
$$\gamma=\sum_{i\geq1}n_i\,|u_i\rangle\langle u_i|$$
for a system $(u_i)$ of orthonormal functions, with $n_i\geq0$ and $\sum_{i\geq1}n_i=\lambda$. Its kinetic energy is defined by (we set $P_j := - \ri \partial_{x_j}$)
$$
    \tr \left( - \Delta \gamma \right)  := \sum_{j=1}^d \tr \left(P_j \gamma P_j\right) = \sum_{j=1}^d \sum_{i \ge 1} n_i \left\| P_j u_i \right\|_{L^2}^2 = \sum_{i\geq1} n_i \int_{\R^d}|\nabla u_i(x)|^2\,\rd x,
$$
and we assume it to be finite. The corresponding density is defined by
$$\rho_\gamma(x):=\sum_{i\geq1} n_i |u_i(x)|^2.$$
This is a non-negative integrable function with $\int_{\R^d}\rho_\gamma(x)\,\rd x=\tr(\gamma)=\lambda$. The Lieb-Thirring inequality for operators~\cite{LieThi-75,LieThi-76,LieSei-09, Frank-20_ppt} states that $\rho_\gamma\in L^{1+2/d}(\R^d)$, with
\begin{equation}
 \|\gamma\|^{\frac2d}\;\tr(-\Delta \gamma) \ge c_{\rm LT}(d)\int_{\R^d}\rho_\gamma(x)^{1+\frac2d}
 \label{eq:LT}
\end{equation}
where $\|\gamma\|=\max (n_i)$ is the operator norm of $\gamma$ on $L^2(\R^d)$. When $\gamma$ is an orthogonal projection, 
\begin{equation}
 \gamma=\sum_{i=1}^N|u_i\rangle\langle u_i|,
 \label{eq:orthonormal_proj}
\end{equation}
this reduces to the inequality mentioned in~\eqref{eq:LT_orthonormal}. The optimal constant $c_{\rm LT}(d)$ is in fact the same as in~\eqref{eq:LT_orthonormal}. On the other hand, the Hoffmann-Ostenhof~\cite{Hof-77} inequality states that
\begin{equation}
\tr(-\Delta \gamma) \geq \int_{\R^d}|\nabla\sqrt{\rho_\gamma}(x)|^2\,\rd x,
\label{eq:Hoffmann-Ostenhof}
\end{equation}
which implies, by the Sobolev inequality, that 
$$\rho_\gamma\in \begin{cases}
(L^1\cap L^\ii)(\R)&\text{when $d=1$,}\\
L^q(\R^2) &\text{for all $1\leq q<\ii$ when $d=2$,}\\
(L^1\cap L^{\frac{d}{d-2}})(\R^d)&\text{when $d\geq3$.}
\end{cases}$$ 
In what follows, we assume that $\| \gamma \| \leq 1$, and we introduce the NLS energy of any such operator $\gamma$ by
\begin{equation}
\boxed{ \cE(\gamma):=\tr(-\Delta \gamma) -\frac1p\int_{\R^d}\rho_\gamma(x)^p\,\rd x.}
 \label{eq:def_energy_gamma}
\end{equation}
We use the same notation $\cE$ as we did for the case of one or $N$ functions, since we think that there cannot be any confusion. When $\gamma=|u\rangle\langle u|$ is a rank-one operator, we recover the usual NLS energy of one function. When $\gamma$ is an orthonormal projection as in~\eqref{eq:orthonormal_proj} we obtain the energy $\cE(u_1,...,u_N)$ introduced in~\eqref{eq:def_energy_u_j}. We introduce the minimisation problem
\begin{equation}
 \boxed{J(\lambda):=\inf\left\{\cE(\gamma)\ :\ 0\leq\gamma=\gamma^*\leq1,\ \tr(\gamma)=\lambda,\ \tr(- \Delta \gamma) <\ii\right\}.}
 \label{eq:def_J_lambda}
\end{equation}
The following well-known result in the spirit of~\cite{Lieb-81} states that $J(N)$ coincides with our previously defined problem for orthonormal functions. This is based on the important fact that the energy is concave in $\gamma$. In fact, using that the map $\gamma \mapsto \Tr(-\Delta \gamma)$ is linear we find
\begin{align*}
    \cE(t \gamma_0 + (1 - t) \gamma_1) = & t \cE(\gamma_0) + (1 - t) \cE(\gamma_1) \\
    & \quad - \frac{1}{p} \int_{\R^d} (t \rho_0 + (1 - t) \rho_1)^p - t \rho_0^p - (1 - t) \rho_1^p,
\end{align*}
where the last integrand is negative by convexity of $x \mapsto x^p$.
Hence, the map $\cE$ always attains its minimum on the extreme points of a convex set.

\begin{lemma}\label{lem:gamma_wants_to_have_finite_rank}
Let $d\geq1$ and $1<p<1+2/d$. The infimum $J(\lambda)$ defined in~\eqref{eq:def_J_lambda} is finite for every $\lambda>0$. Let $N$ be the smallest integer such that $N\geq\lambda$. Then we have
\begin{multline}
J(\lambda)=\inf\bigg\{\cE(\gamma)\ :\ \gamma=\sum_{i=1}^{N-1}|u_i\rangle\langle u_i|+(\lambda-N+1)|u_{N}\rangle\langle u_{N}|,\\ 
(u_1,...,u_{N})\in H^1(\R^d)^{N}\ \text{is orthonormal in $L^2(\R^d)$}\bigg\}.
\label{eq:equality_finite_rank}
\end{multline}
In particular,
\begin{itemize}
 \item[$(i)$] for every integer $N \in \N$, $J(N)$ coincides with~\eqref{eq:def_J_integers},
 \item[$(ii)$] for every $0\leq\lambda\leq1$, we have $J(\lambda)=I(\lambda)$, given by the NLS formula~\eqref{eq:formula_I_lambda}.
\end{itemize}
\end{lemma}

Note that Equation~\eqref{eq:equality_finite_rank} is in fact true for all $p>1$, even in the case $J(\lambda) = - \infty$.

\begin{proof}
From the Lieb-Thirring inequality~\eqref{eq:LT} and Lemma~\ref{lem:LTminimisation} we have 
\begin{equation}
J(\lambda)\geq\inf_{\substack{\rho\geq0\\ \int_{\R^d}\rho=\lambda}}\int_{\R^d}\left(c_{\rm LT}(d)\rho(x)^{1+\frac2d}-\frac1p\rho(x)^p\right)\,\rd x
=   e_{\rm LT}(d, p)\,\lambda,
\label{eq:use_LT_lower_bound}
\end{equation}
where $e_{\rm LT}(d,p)$ is given by~\eqref{eq:e_LT}. In particular, $J(\lambda)$ is finite for all $\lambda\geq0$. 

Let us prove that we can restrict the minimisation problem to finite rank operators. Let $\gamma=\sum_{i\geq1}n_i|u_i\rangle\langle u_i|$ be any admissible operator of infinite rank. Upon relabelling of the indices, we may assume $0 < n_1  < 1$. Set $0 < \varepsilon < 1 - n_1$, and $K_0 \in \N$ so that $\sum_{i \ge K_0+1} n_i < \varepsilon$. Then, for all $K \ge K_0$, the operator
$$\gamma_K:= \left( n_1 + \sum_{i \ge K+1} n_i  \right) |u_{1}\rangle\langle u_{1}|  + \sum_{i=2}^Kn_i|u_i\rangle\langle u_i|$$
is also admissible and converges to $\gamma$ when $K\to\ii$, in the trace norm. We even have 
$$\tr(-\Delta)\gamma_K=\left( n_1 + \sum_{i \ge K+1} n_i  \right) \int_{\R^d}|\nabla u_{1}|^2+ \sum_{i=2}^Kn_i\int_{\R^d}|\nabla u_i|^2\underset{K\to\ii}{\longrightarrow}\tr(-\Delta\gamma).$$
The trace-class convergence implies $\rho_K\to\rho$ in $L^1(\R^d)$ and since $\rho_K$ is bounded in $L^{1+2/d}(\R^d)$ by the Lieb-Thirring inequality~\eqref{eq:LT}, we have $\rho_K\to\rho$ in $L^p(\R^d)$. Hence $\cE(\gamma_K)\to\cE(\gamma)$ and the infimum can be restricted to finite-rank operators, as claimed.

Let now $\gamma=\sum_{i \ge 1} n_i|u_i\rangle\langle u_i|$ be any finite-rank admissible operator with $0 \le n_i \le 1$ and $\sum_{i \ge 1} n_i = \lambda$, and assume that there are two $n_i$, say $n_1$ and $n_2$, which belong to the open interval $(0,1)$. Consider the new operator obtained by varying these two occupation numbers
$$\tilde\gamma_t:=\gamma+t\Big(|u_1\rangle\langle u_1|-|u_2\rangle\langle u_2|\Big),$$
which is admissible as soon as $\max(-n_1,n_2-1)\leq t\leq\min(1-n_1,n_2)$. As we have explained before the lemma, the energy of $\tilde\gamma_t$ is concave in $t$, hence the minimum over $t$ must be attained at the boundary of the interval, where one of the two occupation numbers of $\tilde\gamma_t$ is equal to either 0 or 1. The new operator $\tilde\gamma_t$ has an energy which is lower than or equal to that of $\gamma$ and it has at least one occupation number in the open interval $(0,1)$ less. Arguing by induction we can therefore find an operator $\gamma'$ of the form in~\eqref{eq:equality_finite_rank} so that $\cE(\gamma')\leq\cE(\gamma)$. So the minimisation problem can be restricted to such $\gamma$'s. 

If $\lambda=N$ is an integer, the minimisation set contains only orthogonal projections, and we recover the minimisation problem introduced in~\eqref{eq:def_J_integers}. If $0<\lambda\leq1$ then we can take $\gamma=\lambda|u\rangle\langle u|$. Its energy equals the NLS energy of $\sqrt{\lambda} u$. Therefore the optimum is for $u=Q_\lambda/\sqrt{\lambda}$ and the minimal energy is equal to $I(\lambda)$. If $\lambda=0$ then the set is reduced to $\gamma=0$ and we find $J(0)=0$. 
\end{proof}

\subsection{Binding inequalities and existence of a minimiser}

Here are some standard observations.

\begin{lemma}[Properties of $J(\lambda$] \label{lem:weakBinding}
Let $d\geq1$ and $1 < p < 1 + 2/d$. Then
    \begin{enumerate}
        \item[(i)] \textbf{(lower bound)} We have 
        \[
        e_{\rm LT}(d,p)\, \lambda \le  J(\lambda) < 0
        \]
        for all $\lambda > 0$, where $e_{\rm LT}(d,p)$ is the constant in~\eqref{eq:e_LT}.

        \smallskip
        
        \item[(ii)] \textbf{(sub-additivity)} For all $0 \le \lambda' \le \lambda$, we have
        \[
            J(\lambda) \le J(\lambda') + J(\lambda - \lambda').\label{eq:weak_binding}
         \]
\item[(iii)] \textbf{(monotonicity and continuity)} The function $\lambda\in\R_+\mapsto J(\lambda)$ is decreasing and Lipschitz, hence continuous.

\smallskip

\item[(iv)] \textbf{(concavity)} It is concave on each interval $(N-1,N)$ with $N\in\N^*$. 
         \end{enumerate}
\end{lemma}

\begin{proof}
We have already seen the inequality of \textit{(i)} in~\eqref{eq:use_LT_lower_bound}. We now show that $J(\lambda)$ is negative for $\lambda>0$. Let $\gamma$ be any admissible operator for $J(\lambda)$ and set $\gamma_a=U_a\gamma U_a^*$ where $U_a$ is the dilation unitary operator by the scaling factor $a$, that is, in terms of operator kernels $\gamma_a(x, y) := a^d \gamma(a x, a y)$. We then have $0 \le \gamma_a \le 1$, $\tr(\gamma_a) = \lambda$, and
\begin{equation}
 \cE(\gamma_a) = a^2 \tr(-\Delta \gamma) - \frac{a^{d(p-1)}}{p} \int_{\R^d} \rho_{\gamma}^{p},
 \label{eq:rescaled}
 \end{equation}
    which is negative for $a$ small enough since $d(p-1)<2$.
   
    We turn to the proof of \textit{(ii)}. Take any two operators of the special form in~\eqref{eq:equality_finite_rank}
    $$\gamma_1=\sum_{i=1}^{N-1}|u_i\rangle\langle u_i|+(\lambda'-N+1)|u_{N}\rangle\langle u_{N}|,
    $$
    and
    $$ \gamma_2=\sum_{j=1}^{M-1}|v_j\rangle\langle v_j|+(\lambda-\lambda'-M+1)|v_{M}\rangle\langle v_{M}|$$
with respectively $\lambda'$ and $\lambda - \lambda'$ particles. We then place $\gamma_2$ far away, at a distance $R\gg1$ along the first axis $e_1=(1,0,...,0)$. This can be done by first translating all the functions $v_j$ into $v_j(\cdot-Re_1)$ but then we have to orthonormalise the functions $u_1,\cdots,u_N,v_1(\cdot-Re_1),\cdots,v_M(\cdot-Re_1)$. Although we could have chosen the $u_j$ and $v_j$ with compact support by a density argument, we use here a different reasoning which will be useful later on. We consider the Gram matrix $S_R$ of the family $u_1,\cdots,u_N, v_{1}(\cdot - Re_1),\cdots ,v_M(\cdot - Re_1)$, namely
$$S_R=\begin{pmatrix}
I_N & E^R\\
(E^R)^*&I_M
\end{pmatrix},\qquad E^R_{ij}=\pscal{u_i,v_j(\cdot-Re_1)}.
$$
Since $S_R$ is a Gram matrix, it is hermitian and positive. For $R$ large enough, it is invertible, and we can set
\begin{align*}
& \begin{pmatrix}u_1^{(R)} &  \cdots & u_N^{(R)} & v_1^{(R)} &  \cdots & v^{(R)}_M\end{pmatrix} :=   \\
& \qquad \begin{pmatrix}u_1 & \cdots &  u_N & v_1(\cdot-Re_1) & \cdots & v_M(\cdot-Re_1)\end{pmatrix} (S_R)^{-\frac12}.
\end{align*}
This family forms an $(N+M)$--orthonormal system. This follows from the classical fact that if $\Phi := (\phi_1, \cdots, \phi_L)$ is an independent family of $L$ vectors, then the $L \times L$ matrix $S$ with coefficients $S_{ij} := \bra \phi_i, \phi_j \ket$ is hermitian positive definite, and the family $\Psi := (\psi_1, \cdots, \psi_L)$ defined by
\[
    \Psi := \Phi S^{-1/2}, \quad \text{that is} \quad \psi_i := \sum_{k =1}^L \phi_k (S^{-1/2})_{ki}
\] 
is orthonormal, since
\[
    \bra \psi_i, \psi_j \ket = \sum_{k,l=1}^L (S^{-1/2})_{ik}  (S^{-1/2})_{lj} \underbrace{\bra \phi_k, \phi_l \ket}_{=S_{kl}} = \left(S^{-1/2} S S^{-1/2}\right)_{ij} = \delta_{ij}.
\]
We then introduce the admissible operator
    $$\gamma^{(R)}=\sum_{i=1}^{N}|u^{(R)}_i\rangle\langle u^{(R)}_i|+(\lambda'-N)|u^{(R)}_{N}\rangle\langle u^{(R)}_{N}|+\sum_{j=1}^{M}|v^{(R)}_j\rangle\langle v^{(R)}_j|+(\lambda-\lambda'-M)|v^{(R)}_{M}\rangle\langle v^{(R)}_{M}|$$
    which has the required trace $\tr\big(\gamma^{(R)}\big)=\lambda$. Since the matrix $S_R$ tends to $I_{N+M}$ when $R\to\ii$, we have 
$$
\norm{u^{(R)}_i-u_i}_{H^1(\R^d)}\to0,\qquad \norm{v^{(R)}_j-v_j(\cdot-Re_1)}_{H^1(\R^d)}\to0.
$$
This already proves that
\[
    \Tr(-\Delta \gamma^{(R)}) - \Tr(-\Delta \gamma_1) - \Tr(- \Delta \gamma_2) \xrightarrow[R \to \infty]{} 0.
\]
Similarly, we have $\rho^{(R)}-\rho_1-\rho_2(\cdot-Re_1)\to0$ in $L^1(\R^d)$ and since the three functions are bounded in $L^{1+2/d}(\R^d)$, we deduce by H\"older's inequality that the same holds in $L^p(\R^d)$. For the nonlinear term, we thus have
\[
    \lim_{R \to \infty} \int_{\R^d} (\rho^{(R)})^p  = 
    \lim_{R \to \infty} \int_{\R^d} \left(\rho_1 + \rho_2(\cdot - Re_1) \right)^p =\int_{\R^d}\rho_1^p  + \rho_2^p,
\]
where the last equality can be proved using the density of $C^\infty_c(\R^d)$ in $L^p(\R^d)$. This proves that 
$$\lim_{|R|\to\ii}\cE\big(\gamma^{(R)}\big)=\cE(\gamma_1)+\cE(\gamma_2).$$
Hence $J(\lambda)\leq \cE(\gamma_1)+\cE(\gamma_2)$. After optimising over $\gamma_1$ and $\gamma_2$ using~\eqref{eq:equality_finite_rank} we conclude that $J(\lambda)\leq J(\lambda')+J(\lambda-\lambda')$. 
    
    The monotonicity in \textit{(iii)} follows from the fact that $J(\lambda-\lambda')<0$ by \textit{(i)}, hence $J(\lambda)<J(\lambda')$ for $0<\lambda'<\lambda$. 
    
Let $\lambda,\lambda'\geq0$ and let us prove that $|J(\lambda)-J(\lambda')|\leq C|\lambda-\lambda'|$, that is, $J$ is Lipschitz. If either $\lambda$ or $\lambda'$ vanishes, this is \textit{(i)}. Hence, without loss of generality we may assume that $0< \lambda'\leq \lambda$. Then from the monotonicity of $J$ we have $J(\lambda)\leq J(\lambda')<0$. To get a bound in the other direction we take any trial state $\gamma$ with trace $\tr(\gamma)=\lambda$ and let $\gamma'=\frac{\lambda'}{\lambda}\gamma$. We have 
$$J(\lambda')\leq \cE(\gamma')=\cE\left(\frac{\lambda'}{\lambda}\gamma\right)=\cE(\gamma)-\frac{\lambda-\lambda'}{\lambda}\tr(-\Delta\gamma)+\frac1p\left(1-\left(\frac{\lambda'}{\lambda}\right)^p\right)\int_{\R^d}\rho_\gamma^p.$$
Using $(1 - (1 - x)^p) \le p x$ we deduce that 
\begin{equation}
J(\lambda')\leq \cE(\gamma)-\frac{\lambda-\lambda'}{\lambda}\left(\tr(-\Delta\gamma)-\int_{\R^d}\rho_\gamma^p\right).
\label{eq:estim_Lipschitz1}
\end{equation}
The last term takes the same form as our original problem but with no $1/p$ in front of the nonlinear term. Using the same argument as in \textit{(i)} based on the Lieb-Thirring inequality~\eqref{eq:LT} we have
$$\inf_{\tr(\gamma)=\lambda}\left(\tr(-\Delta\gamma)-\int_{\R^d}\rho_\gamma^p\right)\geq -C\lambda.$$
We have thus shown that 
$$J(\lambda')\leq \cE(\gamma)+C(\lambda-\lambda'),$$
for a constant $C$ depending only on $p$ and $d$. Thus we obtain 
$$J(\lambda)\leq J(\lambda')\leq J(\lambda)+C(\lambda-\lambda')$$
after optimising over $\gamma$.
    
Finally we prove concavity on each interval $(N-1,N)$ with $N\in\N$. Let $N-1\le \lambda_1 < \lambda < \lambda_2 \le N$, and let $t \in (0, 1)$ be such that $\lambda=t\lambda_1+(1-t)\lambda_2$. For any admissible $\gamma$ of the form in~\eqref{eq:equality_finite_rank} with $\tr(\gamma)=\lambda$, we can write $\gamma=t\gamma_1+(1-t)\gamma_2$ with 
$$\gamma_{1,2}=\sum_{i=1}^{N-1}|u_i\rangle\langle u_i|+(\lambda_{1,2}-N+1)|u_{N}\rangle\langle u_{N}|.$$
Since $\cE$ is concave, this implies
$\cE(\gamma)\geq t\cE(\gamma_1)+(1-t)\cE(\gamma_2)\geq tJ(\lambda_1)+(1-t)J(\lambda_2)$. Minimising over $\gamma$ yields the desired concavity. 
\end{proof}

\begin{remark}[Concavity on $\R_+$]\label{rmk:concavity}
The concavity over $\R_+$ is not expected to hold in general. 
Intuitively, for $\lambda\in(N-1,N)$ the derivative $J'(\lambda)$ should be equal to the last (partially) filled eigenvalue $\mu_N$ (see~\eqref{eq:estim_derivative_one_side} below for a one-sided estimate). However at $\lambda=N\in\N$ we expect that $J'(N)_-=\mu_N$ whereas $J'(N)_+=\mu_{N+1}$ which respectively correspond to the last filled eigenvalue when we decrease the mass or to the next eigenvalue to be filled when we increase it. Since $\mu_{N+1}\geq\mu_N$, $J$ is not expected to be concave except when $\mu_N=\mu_{N+1}$. 

A numerical computation in dimension $d=1$ in Figure~\ref{fig:J_over_lambda} below confirms that $J$ is not concave over $\R_+$. Also $\lambda\mapsto J(\lambda)/\lambda$ is not decreasing, except when restricted to integers.  
\end{remark}

\begin{figure}
\includegraphics[width=6.5cm]{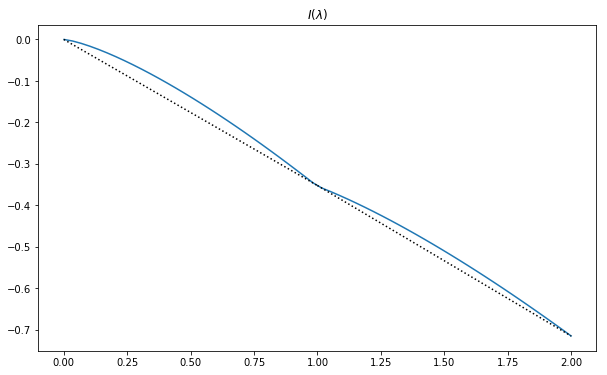}\includegraphics[width=6.5cm]{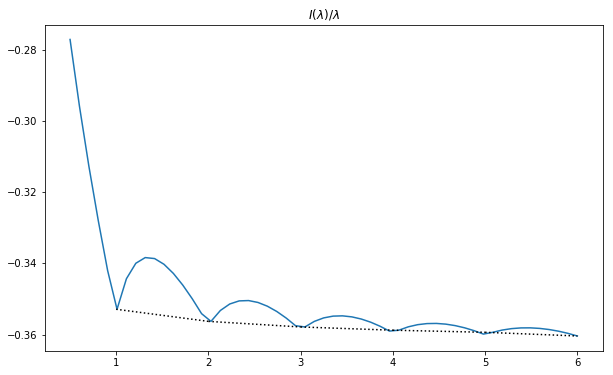}
\caption{Numerical computation of the functions $\lambda\mapsto J(\lambda)$ \emph{(left)} and  $\lambda\mapsto J(\lambda)/\lambda$ \emph{(right)} in dimension $d=1$ with $p=1.3$. The function has the behaviour described in Remark~\ref{rmk:concavity}.
\label{fig:J_over_lambda}}
\end{figure}

The next theorem follows from using the concentration-compactness technique for operators~\cite{Friesecke-03,FraLieSeiSie-07,Lewin-11,LenLew-10} and from the concavity of $J(\lambda)$ on each $(N-1,N)$. It also follows from a profile decomposition similar to that used in~\cite{HonKwoYoo-19}. We will not write the proof in this paper, since the result is in fact also contained in~\cite[Thm. 27]{Lewin-11} (with $W = 0$ there) using Remark~\ref{rmk:N-particle} about the link with the $N$-particle problem.

\begin{theorem}[Existence under the binding condition]\label{th:Lewin}
    Let $d\geq1$ and $1<p<1+2/d$.  Let $N \in \N$ be such that the following binding inequalities hold:
    \begin{equation} \label{eq:strong_binding}
         \forall k = 1, \cdots, N - 1, \quad J(N) < J(k) + J(N-k).
    \end{equation}
    Then, the problem $J(N)$ has a minimiser.
\end{theorem}

In usual concentration compactness theory, one requires the continuous binding inequality
\begin{equation*}
   \forall 0 < \lambda' < N, \quad J(N) < J(\lambda') + J(N-\lambda').
\end{equation*}
These inequalities automatically follow from the integer case~\eqref{eq:strong_binding} because the function $\lambda'\mapsto J(\lambda')+J(N-\lambda')$ is concave over each interval $(k,k+1)$ (Lemma~\ref{lem:weakBinding}), hence its minimum is attained either at $\lambda'=k$ or at $\lambda'=k+1$. 

\begin{remark}[Non-integer case] \label{rem:binding_noninteger}
Let $N\in\N$ and $\alpha\in(0,1)$ be such that 
\begin{equation*} 
        \forall k=1,\cdots,N, \quad J(N+\alpha) <  J(k) + J(N-k+\alpha).
    \end{equation*}
By the concentration-compactness method one can conclude that $J(N+\alpha)$ has a minimiser. Note that $k=N$ is included in the above conditions. 
\end{remark}

\subsection{Properties of minimisers}

Here we state some general properties of minimisers, assuming they exist.

\begin{proposition}[Euler-Lagrange equations] \label{lem:Euler-Lagrange} 
    Let $d\geq1$ and $1<p<1+2/d$. 
    If $J(\lambda)$ admits minimisers then it possesses one which is of the form in~\eqref{eq:equality_finite_rank}, with (orthonormal) real-valued eigenfunctions:
    \begin{equation} \label{eq:sc_gamma}
        \gamma=\sum_{i=1}^{N-1}|u_i\rangle\langle u_i|+(\lambda-N+1)|u_{N}\rangle\langle u_{N}|,
    \end{equation}
    where $N$ is the smallest integer such that $N\geq\lambda$. 
    The $u_i$ are the $N$ first eigenfunctions of the operator $-\Delta-\rho_\gamma^{p-1}$, counted with multiplicity:
\begin{equation}
\left(-\Delta-\rho_\gamma^{p-1}\right)u_i=\mu_i\,u_i,\qquad i=1,...,N 
 \label{eq:fermionic_NLS_2}
\end{equation}
    with $\mu_1<\mu_2\leq\cdots \leq\mu_N<0$. 
    In particular, $-\Delta-\rho_\gamma^{p-1}$ has at least $N$ negative eigenvalues. The functions $u_i$ are real-analytic and tend to zero at infinity. Finally, we have the estimate
    \begin{equation}
    \frac{2p-d(p-1)}{2-d(p-1)}\frac{J(\lambda)}{\lambda}\leq \mu_N\leq J(1)(\lambda-N+1)^{\frac{2}{d}\frac{p-1}{1+\frac2d -p}}<0
    \label{eq:estim_mu_N}
    \end{equation}
    on the last filled eigenvalue.
\end{proposition}

Equation~\eqref{eq:sc_gamma} can also be written in the compact operator form
\[
    \gamma = \1(H_\gamma < \mu_N) + \delta, \quad \text{with} \quad H_\gamma := - \Delta + \rho_\gamma^{p-1},
\]
where $\delta$ is a self-adjoint operator on ${\rm Ker}(H_\gamma - \mu_N)$, which is the sum of a projector plus a rank-one operator.

In the proof of the proposition, we are going to use the following remark which follows from the concavity of $\cE$.

\begin{lemma}[A general inequality]\label{lem:inequality_concavity}
Let $d\geq1$ and $1<p<1+2/d$. 
Let $0\leq \gamma=\gamma^*\leq 1$ and $0\leq \gamma'=(\gamma')^*\leq 1$ be two admissible operators.
Then we have
\begin{equation}
 \cE(\gamma')\leq\cE(\gamma)+\tr \,H_\gamma(\gamma'-\gamma)
 \label{eq:relation_gamma_1_2}
\end{equation}
where 
$$H_\gamma:=-\Delta-\rho_{\gamma}^{p-1}$$
is self-adjoint on $H^2(\R^d)$ and the trace in~\eqref{eq:relation_gamma_1_2} is understood in the quadratic form sense. 
\end{lemma}

\begin{proof}[Proof of Lemma~\ref{lem:inequality_concavity}]
The proof of~\eqref{eq:relation_gamma_1_2} follows from the equality
\begin{equation*}
\cE(\gamma')=\cE(\gamma)+\tr(-\Delta-\rho_{\gamma}^{p-1})(\gamma'-\gamma)-\frac1p\int_{\R^d}\Big(\rho_{\gamma'}^p-\rho_{\gamma}^p-p\rho_{\gamma}^{p-1}
(\rho_{\gamma'}-\rho_{\gamma})\Big),
\end{equation*}
and the fact that the last integrand is non-negative by the convexity of $x \mapsto x^p$. 

Let us prove that $H_\gamma$ is self-adjoint. From~\eqref{eq:Hoffmann-Ostenhof} we have $\rho_\gamma\in L^p(\R^d)$ for $1<p<\ii$ in dimensions $d=1,2$ and for $1<p<1+2/(d-2)$ in dimensions $d\geq3$. In particular, $\rho_\gamma^{p-1}\in L^r(\R^d)$ for all $1/(p-1)<r<\ii$ in dimensions $d=1,2$ and for
   $$\frac1{p-1}<r< \frac{1+\frac{2}{d-2}}{p-1},\qquad\text{where} \quad \frac{1+\frac{2}{d-2}}{p-1}>\frac{d}2\left(1+\frac{2}{d-2}\right)>\max\left(2,\frac{d}2\right)$$
   in dimensions $d\geq3$. From the Rellich-Kato and Weyl theorems~\cite{ReeSim4}, this shows that the operator $H_\gamma=-\Delta-\rho_\gamma^{p-1}$ is self-adjoint on $H^2(\R^d)$ and that its essential spectrum equals $[0,\ii)$.
\end{proof}

With Lemma~\ref{lem:inequality_concavity} at hand we can write the

\begin{proof}[Proof of Proposition~\ref{lem:Euler-Lagrange}]
We split the proof into several steps.
   
\subsubsection*{Step 1: Equation.}
   Let $\gamma$ be a minimiser for $J(\lambda)$. Defining the complex conjugate $\overline{\gamma} := \sum_{i} n_i | \overline{u_i} \ket \bra \overline{u_i} |$ which has the same density $\rho_{\overline\gamma}=\rho_\gamma$, we notice that
   $$\cE(\gamma)=\cE\left(\frac{\gamma+\overline\gamma}{2}\right),$$
   so we may assume that $\overline{\gamma}=\gamma$. The eigenfunctions of $\gamma$ can then be chosen real. 
   
Let $\gamma'$ be any other admissible operator of trace $\lambda=\tr(\gamma')$.  Using~\eqref{eq:relation_gamma_1_2} and the fact that $\cE(\gamma')\geq \cE(\gamma)$ we deduce that 
$$\tr\; H_\gamma(\gamma'-\gamma)\geq0.$$
In other words, $\gamma$ also solves the linear minimisation problem   
   \begin{equation}
   \inf_{\tr(\gamma')=\lambda}\tr \big(H_\gamma\gamma'\big).
    \label{eq:min_linear_pb}
   \end{equation}
   Minimisers of the linear problem~\eqref{eq:min_linear_pb} exist only when $H_{\gamma}$ has at least $N$ non-positive eigenvalues, and are all of the form 
   $$\gamma'=\1_{(-\ii,\mu_N)}\big(H_\gamma\big)+\delta$$
   with $0\leq \delta=\delta^*\leq \1_{\{\mu_N\}}\big(H_\gamma\big)$ and $\delta\neq0$. This is called the \emph{aufbau principle} in quantum chemistry. The eigenvalues are filled starting from the bottom and only the last eigenvalue can be partially filled. Our minimiser $\gamma$ must therefore be of this form. 
   Note that $H_\gamma$ is real since $\rho_\gamma$ is a real function, therefore $\1_{(-\ii,\mu_N)}(H_\gamma)$ is real as well. We conclude that $\delta=\overline\delta$.
   It remains to show that $\delta$ is a projection plus a rank-one operator. In the next step we prove that $\mu_N<0$, which already implies that $\delta$ must be finite rank.
   
 \subsubsection*{Step 2: Estimates on $\mu_N$.}  
   We first show that $\mu_N<0=\min\sigma_{\rm ess}(H_\gamma)$. We consider $\gamma'=\gamma-t|u_N\rangle\langle u_N|$, which is admissible for $0\leq t\leq (\lambda-N+1) \leq1$. From~\eqref{eq:relation_gamma_1_2} we have\footnote{This inequality also implies
       $J'(\lambda)_-=\lim_{t\to0^+}\frac{J(\lambda-t)-J(\lambda)}{-t}\geq \mu_N$, see Remark~\ref{rmk:concavity}.}
\begin{equation}
J(\lambda-t)\leq \cE(\gamma')\leq \cE(\gamma)-\mu_Nt=J(\lambda)-\mu_Nt.
\label{eq:estim_derivative_one_side}
\end{equation}
 Using $J(\lambda-t)\geq J(\lambda)-J(t)$ and the explicit formula for $J(t)=I(t)$ in~\eqref{eq:formula_I_lambda} we obtain the inequality
   \begin{equation*}
   \mu_N\leq \frac{J(\lambda-N+1)}{\lambda-N+1}=J(1)(\lambda-N+1)^{\frac{2}{d}\frac{p-1}{1+\frac2d -p}}<0.
   \end{equation*}   

Next we derive the lower bound~\eqref{eq:estim_mu_N} on $\mu_N$. To this end we use the virial (also called Pohozaev) identity. Let $\gamma_a$ be the rescaled operator as in~\eqref{eq:rescaled}. Then the function
    $$a\mapsto \cE(\gamma_a)=a^2\tr(-\Delta \gamma)-\frac{a^{d(p-1)}}{p}\int_{\R^d}\rho_\gamma^p$$
    must attain its minimum at $a=1$. Writing that the derivative vanishes at this point we find the virial identity
    $$\tr(-\Delta \gamma) =\frac{d(p-1)}{2p}\int_{\R^d}\rho_\gamma^p.$$
    This gives
    $$J(\lambda)=\cE(\gamma)=\frac{d(p-1)-2}{2p}\int_{\R^d}\rho_\gamma^p
        \quad \text{and} \quad 
        \tr(H_\gamma\gamma)=\frac{d(p-1)-2p}{2p}\int_{\R^d}\rho_\gamma^p,
    $$
    so that 
    \begin{equation}
    \frac{2p-d(p-1)}{2-d(p-1)}J(\lambda)=\tr(H_\gamma\gamma)=\sum_{i=1}^{N-1}\mu_i+(\lambda-N+1)\mu_N \le \lambda \mu_N,
    \label{eq:Virial}
    \end{equation}
    where the last inequality comes from the fact that $\mu_i\leq\mu_N$ for all $i$. We obtain as claimed
$$\mu_N\geq \frac{2p-d(p-1)}{2-d(p-1)}\frac{J(\lambda)}{\lambda}.$$    

\subsubsection*{Step 3: Regularity and decay.}  
Note that the first eigenfunction $u_1$ of $H_\gamma$ is always positive and non-degenerate. Therefore $\rho_\gamma>0$. Since we now have a system of finitely many coupled Partial Differential Equations, the real-analyticity of the $u_i$'s follows from classical results~\cite{Morrey-58,Kato-96}. 
  
Next we show that the functions $u_i$ tend to 0 at infinity. In dimension $d\leq3$, this follows from the fact that $u_i\in H^2(\R^d)$ as we have seen in Lemma~\ref{lem:inequality_concavity}. In dimensions $d\geq4$ we need to employ a simple boot-strap argument. Assuming that $u_i\in L^r(\R^d)$ for all $i=1,...,N$, we infer that $\rho\in L^{r/2}(\R^d)$ and then $(-\Delta+|\mu_i|)u_i=\rho^{p-1}u_i\in L^{\frac{r}{2p-1}}(\R^d)$. When $r<d(2p-1)/2$, the Sobolev embedding shows that $u_i\in L^{f(r)}(\R^d)$ with $f(r)=dr/(2dp-d-2r)$. This function $f$ has the two fixed points $r=0$ and $r=d(p-1)<2$, the latter being unstable. Starting from any $r_0>2$ we obtain after iterating $f$ finitely many times an $r>d(2p-1)/2$. The Sobolev-Morrey embeddings now prove that the $u_i$ are continuous and tend to 0 at infinity.

\subsubsection*{Step 4. Form of $\delta$}
To prove that $\delta=\overline\delta$ is a finite rank projection plus a rank-one operator, we assume by contradiction that $\delta$ has two eigenvalues $\delta_1,\delta_2\in(0,1)$ with corresponding orthonormal real-valued eigenfunctions $u_i,u_j$ and we vary the corresponding eigenvalues linearly like $\delta_1+t,\delta_2-t$, as we did in the proof of Lemma~\ref{lem:gamma_wants_to_have_finite_rank}. The energy is concave in $t$, hence must be constant since $\gamma$ is a minimiser. The nonlinear term is even strictly concave, unless $|u_i|=|u_j|$. Since these are real-analytic real-valued functions, it would imply $u_i=\pm u_j$ everywhere, a contradiction. Therefore, at most one eigenvalue of $\delta$ can be in $(0,1)$ and this concludes the proof of Proposition~\ref{lem:Euler-Lagrange}. 
\end{proof}

\begin{remark}
The upper bounds on $\mu_N$ in~\eqref{eq:estim_mu_N} deteriorates when $\lambda\to(N-1)^+$. We were not able to bound $J'(N)_+$, as it probably requires the evaluation of $\mu_{N+1}(\lambda)$, which may vanish as $\lambda \to N^+$.
\end{remark}

Using the Euler-Lagrange equation~\eqref{eq:fermionic_NLS_2}, we can prove that the functions $u_i$ (and therefore the density $\rho$) are exponentially decaying. Actually, we may provide lower bounds as well. This is not obvious, because the functions $u_i$ have non trivial nodal sets for $i\geq2$. Only $u_1$ is positive everywhere. Following~\cite{BarMer-77,HofHofSwe-85}, we introduce, for $f \in L^q_{\rm loc}(\R^d, \C)$, the $q$-spherical average 
\begin{equation} \label{eq:average_f}
    [f]_q(x) := \left( \dfrac{1}{| \SS^{d-1} |} \int_{\SS^{d-1}} \big| f\big(|x| \omega\big) \big|^q\, \rd \sigma(\omega) \right)^{\frac1q}=\left( \int_{{\rm SO}(d)} \big| f\big(\cR x\big) \big|^q\, \rd \cR \right)^{ \frac1q}.
\end{equation}
In the second integral we use the normalised Haar measure on SO$(d)$. 

\begin{lemma}[Decay of minimisers at infinity] \label{lem:exp_decay_minimiser}
    Let $d\geq1$ and $1<p<1+2/d$. Let $\gamma$ be a real minimiser of $J(\lambda)$ of the form~\eqref{eq:equality_finite_rank}, with density $\rho = \rho_\gamma$.
Then we have the bounds 
    \begin{equation} \label{eq:estimates_rho}
       \frac1C \dfrac{\re^{ - 2 \sqrt{|\mu_N|} |x|}}{1 + |x|^{d-1}} \le [\rho]_1(x) 
        \quad \text{and} \quad 
        \rho(x)+|\nabla\rho(x)|  \le C \dfrac{\re^{ - 2 \sqrt{|\mu_N|} |x|}}{1 + |x|^{d-1}}
    \end{equation}
    for some constant $C>0$. 
    Similarly, for the eigenfunctions $u_i$ of $-\Delta-\rho^{p-1}$ with eigenvalue $\mu_i$ as in~\eqref{eq:fermionic_NLS_2}, we have
    \begin{equation} \label{eq:estimates_ui}
        \frac1{C_i} \dfrac{\re^{ - \sqrt{|\mu_i|} |x|}}{1 + |x|^{\frac{d-1}{2}}} \le [u_i]_2(x) 
        \quad \text{and} \quad 
        | u_i (x)|+|\nabla u_i(x)|  \le C_i \dfrac{\re^{ - \sqrt{|\mu_i|} |x|}}{1 + |x|^{\frac{d-1}{2}}}.
    \end{equation}
\end{lemma}

The slowest exponential is the one corresponding to the $N$th eigenvalue $\mu_N$ and it is the leading term in the density $\rho$. 

\begin{proof}[Proof of Lemma~\ref{lem:exp_decay_minimiser}]
First we derive a non-optimal exponential bound on $\rho$. We have 
    \begin{align*}
        - \Delta \rho & = 2 \sum_{i=1}^N \left( u_i ( - \Delta u_i) - | \nabla u_i |^2   \right)
             = 2 \sum_{i=1}^N \left( \mu_i | u_i |^2 + \rho^{p-1} | u_i |^2 -  | \nabla u_i |^2 \right) \\
             & \le 2 (\mu_N  +  \rho^{p-1})\rho,
    \end{align*}
    where we used the fact that $\mu_i \le \mu_N < 0$ in the last inequality. Since $\rho$ goes to $0$ at infinity and $p>1$, there exists a large enough $R > 0$ such that $\rho^{p-1}(x) < \frac12 |\mu_N| $ for all $| x | \geq R$. In particular, we have
    \[
        \left( - \Delta + |\mu_N| \right) \rho(x) \le 0, \quad \forall | x | \geq R.
    \]
Let $Y_m(x):=\sqrt{m}\,|x|^{1-d/2} K_{\frac{d-2}{2}}(m|x|)>0$ be the Yukawa potential, solution to 
$( - \Delta + m^2) Y_m = 0$ on $\R^d\setminus\{0\}$. Here $K_\alpha$ is the modified Bessel function of the second kind. 
From the asymptotic behavior of $K_\alpha$~\cite[9.7.2]{AbramowitzStegun} we have 
\begin{equation}
Y_m(x) \underset{|x|\to\ii}\sim \sqrt{\frac\pi2} \dfrac{\re^{- m | x |}}{|x|^{\frac{d-1}{2}}}.
\label{eq:behaviour_Yukawa} 
\end{equation}
On the sphere $R \SS^{d-1}$ of radius $R$, the function $\rho$ is bounded, so we have $\rho(x) \le C Y_m(x)$ for some $C>0$ with $m=\sqrt{|\mu_N|}$. Since 
    $$(-\Delta + | \mu_N |) (\rho- CY_m)=(-\Delta + | \mu_N |) \rho \le 0\qquad\text{on $\R^d \setminus B_R$},$$ 
    we deduce from the maximum principle~\cite[Chapter 9.4]{LieLos-01} that $\rho(x) \le C Y_m(x)$ on $\R^d \setminus B_R$. The function $\rho$ is bounded on $B_R$ and therefore we have proved the pointwise upper bound 
\begin{equation}
        \forall x \in \R^d, \quad \rho(x) \le C' \dfrac{\re^{ - \sqrt{|\mu_N|} |x|  }}{1 + |x|^{\frac{d-1}{2}}}.
        \label{eq:pointwise_exp_1}
\end{equation}
Compared with~\eqref{eq:estimates_rho}, we see that a factor $2$ is missing in the exponential. However we have learned that the potential $-\rho^{p-1}$ is exponentially decaying at infinity and for proving~\eqref{eq:estimates_ui} we can now rely on existing results for the eigenfunctions $u_i$ of the linear Schr\"odinger operator $-\Delta-\rho^{p-1}$. 
    
The upper bound in~\eqref{eq:estimates_ui} is well known but we give a detailed proof for completeness. Let for instance $F_i(x):=\sqrt{m_i}\,|x|^{1-\frac{d}2} K_{\frac{d-2}{2}+\eps}(m_i|x|)$ with $m_i=\sqrt{|\mu_i|}$ which solves, this time,
\begin{equation}
\left( - \Delta + |\mu_i|+\frac{\eps(d-2+\eps)}{|x|^2}\right) F_i = 0\qquad\text{on $\R^d\setminus\{0\}$}
\label{eq:F_i}
\end{equation}
see for instance~\cite{DerRic-17}. 
We choose $\eps$ small enough so that the inverse-square potential is attractive: $\eps(d-2+\eps)<0$. In other words, we take $\eps>0$ in dimension $d=1$ and $\eps<0$ in dimension $d\geq3$. The function $F_i$ is then positive on $\R^d\setminus\{0\}$. In dimension $d=2$ we have to take $\eps\in i\R$ and the resulting radial function $F_i$ vanishes infinitely many times close to the origin but it is positive for $|x|\geq |\eps|$~\cite{FerSes-70}. In all cases we have again
    \[
F_i(x) \underset{|x|\to\ii}\sim \sqrt{\frac\pi2} \dfrac{\re^{- \sqrt{|\mu_i|} | x |}}{|x|^{\frac{d-1}{2}}}
    \]
    by~\cite[9.7.2]{AbramowitzStegun}.
Using Kato's inequality~\cite[Theorem~X.27]{ReeSim2} and the exponential decay of $\rho^{p-1}$ we find 
$$\left(-\Delta+\frac{\eps(d-2+\eps)}{|x|^2}+|\mu_i|\right)|u_i|\leq 0 \qquad\text{on $\R^d\setminus B_R$}$$
for $R$ large enough. The maximum principle then gives as before $|u_i|\leq CF_i$ on $\R^d\setminus B_R$, and thus the upper bound on $|u_i|$ in~\eqref{eq:estimates_ui}. For the estimate on $\nabla u_i$, we can use that 
$$u_i=(-\Delta+|\mu_i|)^{-1}(\rho^{p-1}u_i)=CY_{m_i}\ast(\rho^{p-1}u_i)$$
with $m_i:=\sqrt{|\mu_i|}$, so that 
$$|\nabla u_i|\leq C|\nabla Y_{m_i}|\ast(\rho^{p-1}|u_i|).$$
The function $|\nabla Y_{m_i}|$ is integrable in a neighborhood of the origin and it behaves like $\sqrt{|\mu_i|}Y_{m_i}$ at infinity, due to the fact that 
\begin{equation}
K_\alpha'(r)\underset{r\to\ii}\sim-K_\alpha(r),
\label{eq:behaviour_K_prime}
\end{equation}
see~\cite[9.7.4]{AbramowitzStegun}. Thus, using our bound on $|u_i|$ and the exponential bound~\eqref{eq:pointwise_exp_1} for $\rho^{p-1}$, we find
$$|\nabla u_i|\leq C(|\nabla Y_{m_i}|\1_{B_1}+Y_{m_i})\ast(\rho^{p-1}|u_i|)\leq C(|\nabla Y_{m_i}|\1_{B_1})\ast Y_{M_i}+CY_{m_i}\ast Y_{M_i}$$
with $M_i:=(p-1)\sqrt{|\mu_N|}+\sqrt{|\mu_i|}>m_i$. The first term on the right behaves like $Y_{M_i}=o(Y_{m_i})$ at infinity, since $|\nabla Y_{m_i}|\1_{B_1}$ is integrable and has compact support. For the second term we remark that the Fourier transform of $Y_{m_i}\ast Y_{M_i}$ is proportional to
$$\frac{1}{(|k|^2+m_i^2)(|k|^2+M_i^2)}=\frac1{M_i^2-m_i^2}\left(\frac{1}{|k|^2+m_i^2}-\frac{1}{|k|^2+M_i^2}\right)$$
and hence 
\begin{equation} 
Y_{m_i}\ast Y_{M_i}=C\frac{Y_{m_i}- Y_{M_i}}{M_i^2-m_i^2}\leq \frac{C}{M_i^2-m_i^2}Y_{m_i}.
\label{eq:convolution_Yukawa}
\end{equation}
The pointwise bound on $|\nabla u_i|$ in~\eqref{eq:estimates_ui} follows from the behaviour~\eqref{eq:behaviour_Yukawa} of $Y_{m_i}$ at infinity and the fact that $u_i$ is $C^\ii$, hence $|\nabla u_i|$ is bounded on compact sets. 

For the lower bound in~\eqref{eq:estimates_ui} this is more complicated and we just apply~\cite[Corollary 2.2]{HofHofSwe-85}. This result exactly states that if $u_i \in H^1(\R^d)$ is solution to $ ( - \Delta + V) u_i = \mu_i u_i$ with $\mu_i < 0$ and $|V|\leq C|x|^{-1-\eps}$ for some $\eps>0$ at infinity, then the lower bound in~\eqref{eq:estimates_ui} holds. Thus we have proved all the estimates in~\eqref{eq:estimates_ui}.  

Finally, we come back to~\eqref{eq:estimates_rho} and notice that 
    \[
        \rho (x) = \sum_{i=1}^N u_i (x)^2,\qquad         [ \rho ]_1(x) = \sum_{i=1}^N [ u_i ]_2(x)^2\geq [ u_N]_2(x)^2.
    \]
Since $0 < | \mu_N | \le | \mu_{N-1} | \le \cdots < | \mu_1 |$, the leading term in the first sum is the exponential involving $| \mu_N |$ and~\eqref{eq:estimates_rho} follows.
\end{proof}

\subsection{Proof of binding}\label{sec:proof_binding}

We now focus on the proof of the binding inequality $J(\lambda+\lambda')<J(\lambda)+J(\lambda')$. The usual proof is to consider minimisers for $J(\lambda)$ and $J(\lambda')$, and to construct from them a good candidate for $J(\lambda+\lambda')$ by putting these two minimisers far from each other. In our case, all quantities are exponentially decaying, which makes the evaluation of the interaction quite delicate. The following is the heart of the paper. 

\begin{proposition}[Exponentially small binding] \label{lem:strongBinding}
Let $d\geq1$, $1<p<1+2/d$ and $\lambda,\lambda'>0$. Assume $J(\lambda)$ and $J(\lambda')$ admits the respective minimisers $\gamma$ and $\gamma'$, satisfying the properties in Proposition~\ref{lem:Euler-Lagrange}. Let $\mu=\mu_N$ and $\mu'=\mu'_M$ be the associated last filled eigenvalue of $H_{\gamma}$ and $H_{\gamma'}$. Then, under the additional condition
    \begin{equation} \label{eq:cond_on_p}
        1<p < 1 + \sqrt{\frac{\min \{| \mu | , |\mu' | \} }{ \max \{ | \mu |, | \mu' | \} }},
    \end{equation}
    we have the binding inequality $J(\lambda+\lambda') < J(\lambda) + J(\lambda')$.
\end{proposition}

\begin{proof}
We only write the proof in the integer case $\lambda=N\in\N$ and $\lambda'=M\in\N$ for clarity. The arguments are exactly the same in the non-integer case, but the notation is a bit more heavy due to the additional rank-one operator. 

Let $\tilde\gamma := \sum_{i=1}^{N} | \tilde u_i \ket \bra \tilde u_i |$  and $\tilde\gamma' := \sum_{j=1}^{M} | \tilde v_j \ket \bra \tilde v_j |$ be two real minimisers for $J(N)$ and $J(M)$ respectively. Recall that our problem is invariant under rotations. Thus we can introduce $u_i(x)=\tilde u_i(\cR x)$ and $v_j(x)=\tilde v_j(\cR'x)$ for some $\cR,\cR'\in{\rm SO}(d)$ and deduce that these are minimisers of $J(N)$ and $J(M)$ as well. We will need to choose $\cR$ and $\cR'$ appropriately. For clarity, we prefer to postpone this discussion to Lemma~\ref{lem:estimates_eR_IR} below. Note that the functions $u_i,v_j$ satisfy the exponential bounds~\eqref{eq:estimates_ui} uniformly in $\cR,\cR'$ since those are invariant under rotations.

Next, we place the second system far away. For $R>0$, we set $v_{j,R}(x) := v_j(x - R e_1)$ where $e_1=(1,0,...,0)$, and we introduce the Gram matrix
    \[
       S_R := \begin{pmatrix}
        \bbI_N & E^R \\
        (E^R)^* & \bbI_M
       \end{pmatrix}, \quad \text{with} \quad E^R_{ij} := \bra u_i, v_{j,R} \ket = \int_{\R^d} u_i(x) v_j(x - Re_1) \rd x,
     \]
     as we did in the proof of Lemma~\ref{lem:weakBinding}.
Since the functions $u_i$ and $v_j$ are real-valued and exponentially decaying, $E^R$ is real and goes to $0$ exponentially fast. So the Gram matrix $S_R$ is real symmetric positive definite for $R$ large enough. As before, the frame
    $$\begin{pmatrix}
       \psi_{1,R} & \cdots & \psi_{N+M,R}
      \end{pmatrix} :=
      \begin{pmatrix}u_1 & \cdots & u_N & v_{1,R} & \cdots & v_{M,R}\end{pmatrix}  (S_R)^{-\frac12}
      $$
is orthonormal. Our trial state is the orthogonal projection onto this frame, given by
\begin{align*}
\gamma_R &= \sum_{i=1}^{N} | \psi_{i,R} \ket \bra \psi_{i,R} |+\sum_{k=1}^{M} | \psi_{N+k,R} \ket \bra \psi_{N+k,R} |
\\
&=\sum_{i,j=1}^N(S_R^{-1})_{ij}|u_i\rangle\langle u_j|+\sum_{k,\ell=1}^{M}(S_R^{-1})_{N+k,N+\ell}|v_{k,R}\rangle\langle v_{\ell,R}|\\
&\qquad +\sum_{i=1}^N\sum_{k=1}^M\Big((S_R^{-1})_{i ,N+k} |u_i\rangle\langle v_{k,R}|+(S_R^{-1})_{N+k,i} |v_{k,R}\rangle\langle u_i|\Big).
\end{align*}
    To compute $\cE(\gamma_R)$, we consider the Taylor expansion with respect to the largest overlap
$$e_R:=\max_{i,j}\int_{\R^d} |u_i(x)|\,|v_j(x - Re_1)| \rd x.$$
Note that $\| E^R \| \le e_R$. We compute all quantities to the order $O(e_R^2)$.

First, from $(1 + E)^{-1} = 1 - E + O(E^2)$, we have
    \[
        (S_R)^{-1} = \begin{pmatrix}
        \bbI_N & 0 \\
        0 & \bbI_M
        \end{pmatrix} - \begin{pmatrix}
        0 & E^R \\
        (E^R)^* & 0
        \end{pmatrix} + O(e_R^2).
    \]
    This gives, to first order and with $\gamma'_R(x,y)=\gamma'(x-Re_1,y-Re_1)$ the translation of $\gamma'$,
    \begin{equation} \label{eq:Taylor_gamma}
        \gamma_R  = \gamma + \ \gamma'_{R} - \sum_{i=1}^N\sum_{j=1}^M E_{ij}^R \left( | u_i \ket \bra v_{j,R} | + | v_{j,R} \ket \bra u_i | \right)  + O_{\|\cdot\|_{1,1}}(e_R^2)
    \end{equation}
    where $\|\gamma\|_{1,1}=\tr|\sqrt{1-\Delta}\gamma\sqrt{1-\Delta}|$ is the Sobolev-type trace norm. Let us evaluate the different terms in the energy. For the kinetic energy, we obtain (recall that everything is real-valued and that $N$ and $M$ are finite)
    \begin{equation*}
         \Tr(- \Delta \gamma_R) - \Tr(- \Delta \gamma) - \Tr(- \Delta \gamma') = - 2 \left(\sum_{i=1}^N\sum_{j=1}^M E_{ij}^R \int_{\R^d} \nabla u_i \cdot \nabla v_{j,R} \right)  + O(e_R^2).
    \end{equation*}
    Using the Euler-Lagrange equations~\eqref{eq:fermionic_NLS}, we see that
    \[
        \big(- \Delta - \rho^{p-1}\big) u_i = \mu_i u_i \quad \text{and} \quad \big(- \Delta- (\rho'_{R})^{p-1} \big)v_{j,R} = \mu'_j v_{j,R},
    \]
    with $\rho=\rho_\gamma$ and $\rho'_R=\rho_{\gamma'}(\cdot-Re_1)$.  This gives for instance
    \[
        \int_{\R^d} \nabla u_i \cdot \nabla v_{j,R} = \mu_i  E_{ij}^R+\int_{\R^d} \rho^{p-1} u_i v_{j,R}.
    \]
Since $\rho$ is bounded, together with the definition of $e_R$, we deduce that
\[
    \Tr(- \Delta \gamma_R) - \Tr(- \Delta \gamma) - \Tr(- \Delta \gamma') =  O(e_R^2).
\]
   
   We now compute the difference for the term $\int \rho^p$. We first find an expression for $\rho_R$, the density of $\gamma_R$. From~\eqref{eq:Taylor_gamma}, we get 
   \begin{equation*}
     \rho_R = \rho + \rho'_{R} -  2  \sum_{i=1}^N\sum_{j=1}^M E_{ij}^R u_i v_{j,R} + O_{L^p(\R^d)}( e_R^2).
   \end{equation*}
   This gives
   \begin{align*}
       \int_{\R^d} \rho_R^p  = &\int_{\R^d} \left( \rho + \rho'_{R} \right)^p
        - 2p \int_{\R^d} \left( \rho + \rho'_{R} \right)^{p-1}   \sum_{i=1}^N\sum_{j=1}^M E_{ij}^R u_i v_{j,R}   + O(e_R^2).
   \end{align*}
  Again, using that $\rho$ and $\rho_R'$ are bounded functions, the last integral is of order $O(e_R^2)$. Altogether, this proves that
   \begin{align*}
J(N+M)-J(N)-J(M)&\leq \cE(\gamma_R) - J(N) - J(M)\\
&=  - \frac1p\int_{\R^d} \Big( (\rho + \rho'_{R})^p  - \rho^p - (\rho'_{R})^p  \Big) + O(e_R^2).
    \end{align*}
The orthonormalisation procedure generates an error of the order $O(e_R^2)$ in the energy. The first term of the second line is the nonlinear interaction and it is always negative, from the concavity of $x \mapsto -x^p$. The question is whether it wins over the error term $O(e_R^2)$. This is the topic of the next result. We recall that our functions $u_i$ and $v_j$ depend on the two rotations $\cR$ and $\cR'$ which we have introduced at the beginning of the proof and which we choose now, in terms of $R$.

    \begin{lemma} \label{lem:estimates_eR_IR}
        Let $\eps := \sqrt{|\mu_N|}$ and $\eps' :=\sqrt{|\mu_M|}$ and assume, without loss of generality, that $\eps' \le \eps$. Then, there is $C \ge 0$ so that for all $R$ large enough
        \begin{equation} \label{eq:bound_eR}
            e_R \le CR^{\frac{3-d}{2}} \re^{ - \varepsilon' R},
        \end{equation}
uniformly in $\cR,\cR'\in{\rm SO}(d)$. On the other hand, for every $R$ large enough there exists $\cR,\cR'\in{\rm SO}(d)$ such that 
        \begin{equation}
            I_R:=\int_{\R^d} \Big( (\rho + \rho'_{R})^p  - \rho^p - (\rho'_{R})^p  \Big) \ge C R^{-p(d - 1)} \re^{ - 2 p \frac{\varepsilon \varepsilon'}{\varepsilon + \varepsilon' }R}.
            \label{eq:lower_bd_I_R}
        \end{equation}
In particular, if $1 < p < 1 + \frac{\varepsilon'}{\varepsilon}$, then $e_R^2 = o(I_R)$.
    \end{lemma}

The polynomial factors $R^{\frac{3-d}{2}}$ and $R^{-p(d-1)}$ in~\eqref{eq:bound_eR} and~\eqref{eq:lower_bd_I_R} are not necessarily optimal but will suffice for our argument. Before we provide the proof of Lemma~\ref{lem:estimates_eR_IR}, we remark that it immediately gives the strict inequality
$$J(N+M)-J(N)-J(M)<0$$
after taking $R$ large enough, whenever $1 < p < 1 + \frac{\varepsilon'}{\varepsilon}$. This condition is equivalent to the one in~\eqref{eq:cond_on_p}. 
It thus only remains to provide the proof of Lemma~\ref{lem:estimates_eR_IR}.
    
\begin{proof}[Proof of Lemma~\ref{lem:estimates_eR_IR}]
Let us first bound $e_R$. By Lemma~\ref{lem:exp_decay_minimiser} we have $|u_i|\leq C Y_{m_i}$ and $|v_j|\leq C Y_{m'_j}$ with $m_i=\sqrt{|\mu_i|}$ and $m'_j=\sqrt{|\mu'_j|}$. Thus we find 
$$\int_{\R^d} | u_i | | v_{j,R} | \leq C Y_{m_i}\ast Y_{m'_j}(Re_1)\leq C Y_{\eps}\ast Y_{\eps'}(Re_1).$$
When $\eps'<\eps$ we have $Y_{\eps}\ast Y_{\eps'}\leq C(\eps-\eps')^{-1}Y_{\eps'}$ by~\eqref{eq:convolution_Yukawa}. When $\eps'=\eps$ we can use that $Y_\eps\ast Y_\eps=C\eps^{-1}\partial_m Y_m|_{m=\eps}$ which behaves at infinity like $rY_\eps(r)$, by~\eqref{eq:behaviour_K_prime}. This gives the bound~\eqref{eq:bound_eR}.

    Let us now bound $I_R$ from below. Recall that $\rho(x)=\tilde\rho(\cR x)$ and $\rho'(x)=\tilde\rho'(\cR'x)$
    From Lemma~\ref{lem:exp_decay_minimiser}, the spherical average of $\tilde \rho$ is bounded below in terms of $(1+|x|)^{1-d} \re^{ - 2 \eps |x|}$. This means that for any large enough $r$, there exists at least one $x_r$ on the sphere of radius $r$ such that $\tilde \rho(x_r)\geq cr^{1-d} \re^{ - 2 \eps r}$. The same lemma provides the pointwise upper bound $|\nabla\tilde \rho(x)|\leq Cr^{1-d} \re^{ - 2 \eps r}$ and this implies that 
$$\tilde \rho(x)\geq \frac{c\,\re^{ - 2 \eps r} }{2r^{d-1}}   $$
on the ball $B(x_r,\eta)$ with $\eta=c/(2C)$. The same property holds for $\tilde \rho'$ at a point $x'_r$. We choose the rotations $\cR$ and $\cR'$ to align the points $x_r$ and $x'_{r'}$ with the radii $r+r'=R$ chosen to obtain the largest possible interaction. More precisely, we introduce the point 
$$x^*=\frac{\eps'}{\eps + \eps'}Re_1$$
where $|x|^{1-d} \re^{ - 2 \eps |x|}$ and $|x-Re_1|^{1-d} \re^{ - 2 \eps |x-Re_1|}$ are of the same order $R^{1-d} \re^{ - 2 \frac{\eps\eps'}{\eps+\eps' }R}$. We then take $\cR,\cR'$ such that 
$$x_r=\cR x^*,\qquad x'_{r'}=\cR'(x^*-Re_1),\qquad r=\frac{\eps'}{\eps + \eps'}R,\qquad r'=\frac{\eps}{\eps + \eps'}R.$$
We obtain 
$$\rho(x) \ge c' R^{1-d} \re^{ - 2 \frac{\eps\eps'}{\eps+\eps' }R},\qquad \rho'_R(x) \ge c' R^{1-d} \re^{ - 2 \frac{\eps\eps'}{\eps+\eps' }R},$$   
for all $x\in B(x^*,\eta)$, with 
$$c'=\frac{c}{2}\left(\frac{\eps}{\eps+\eps'}\right)^{1-d}.$$
Since the function $(x,y)\mapsto (x+y)^p-x^p-y^p$ is increasing in $x$ and in $y$ separately for $p>1$, we deduce that 
   \[
       \left[ \big(\rho + \rho'_R\big)^p  - \rho^p - (\rho'_R)^p  \right](x) \ge (c')^p (2^p - 1) R^{p(1-d)} \re^{ - 2 p \frac{\eps\eps'}{\eps+\eps'}R},
   \]
   for all $x\in\ B(x^*,\eta)$. From the positivity of the integrand on $\R^d$ we thus obtain 
   $$I_R\geq \int_{B(x^*,\eta)}\left[ \big(\rho + \rho'_R\big)^p  - \rho^p - (\rho'_R)^p  \right] \ge |B(0,\eta)|(c')^p (2^p - 1) R^{p(1-d)} \re^{ - 2 p \frac{\eps\eps'}{\eps+\eps'}R}$$
   which is the claimed lower bound~\eqref{eq:lower_bd_I_R}.
    \end{proof}

    This concludes the proof of Proposition~\ref{lem:strongBinding}.    
\end{proof}

One immediate consequence of Proposition~\ref{lem:strongBinding} is the case $\lambda = \lambda'$, where the condition~\eqref{eq:cond_on_p} is always satisfied whenever $p < 2$. Since $J(1)=I(1)$ always has a minimiser, we also conclude that $J(2)$ always has one, for $p<2$. 

\begin{corollary} \label{cor:J(2N)}
Let $d\geq1$ and $1 < p  < \min \{ 2, 1+{2}/{d} \}$. Then, if $J(\lambda)$ has a minimiser, we have $J(2\lambda) < 2 J(\lambda)$. In particular, $J(2)<2J(1)$ and $J(2)$ has a minimiser. 
\end{corollary}

\subsection{Proof of Theorem~\ref{thm:ground_states_orthonormal_small_p}}\label{sec:proof_existence_small_p}

Assume that $J(N)$ and $J(M)$ have minimisers. Let $\mu_N$ and  $\mu'_M$ be the corresponding last filled eigenvalues. From~\eqref{eq:estim_mu_N} in the integer case $\lambda = N$, and the fact that $0 > J(N)/N\geq e_{\rm LT}(d,p)$ by \textit{(i)} in Lemma~\ref{lem:weakBinding}, we have 
\begin{equation} \label{eq:pc(d)}\sqrt{\frac{\min\{|\mu_N|,|\mu'_M|\}}{\max\{|\mu_N|,|\mu'_M|\}}}
\ge \sqrt{\frac{|I(d,p,1)|}{|e_{\rm LT}(d,p)|}}\sqrt{\frac{2-d(p-1)}{2p-d(p-1)}}.
\end{equation}
Let $p_c(d)\in(1,2)$ be the first zero of the function
$$p\mapsto 1+\sqrt{\frac{|I(d,p,1)|}{|e_{\rm LT}(d,p)|}}\sqrt{\frac{2-d(p-1)}{2p-d(p-1)}}-p.$$
Using that $e_{\rm LT}(d,p)$ and $I(d,p,1)$ have a finite limit when $p\to1^+$ we deduce that $p_c(d)>1$. 
Hence, if $J(N)$ and $J(M)$ have minimisers and if $1<p<p_c(d)$, then $J(N+M)<J(N)+J(M)$ by Proposition~\ref{lem:strongBinding}. 

Since we already know that $J(1)=I(1)$ has a minimiser, we can deduce by induction on $N$ that binding holds and that $J(N)$ has a minimiser for all $N$, when $1<p< p_c(d)$. This concludes the proof of Theorem~\ref{thm:ground_states_orthonormal_small_p}.\qed

\begin{remark}[Numerical evaluation of $p_c(d)$]\label{rmk:comput_p_c}
Any lower bound on  $|I(d,p,1)|$ and on the Lieb-Thirring constant $c_{\rm LT}(d)$ appearing in the definition~\eqref{eq:e_LT} of $e_{\rm LT}(d,p)$ yields a lower bound on $p_c(d)$. The NLS energy $I(d,p,1)$ can easily be computed numerically to a high precision since this amounts to solving an ordinary differential equation for the radial function $Q$. Using a Runge-Kunta solver for $Q$ and the recent lower bound 
\begin{equation}
 c_{\rm LT}(d)\geq(0.471851)^{\frac1d}\,c_{\rm sc}(d)
 \label{eq:lower_bd_LT}
\end{equation}
 from~\cite[Prop.~10]{FraHunJexNam-18_ppt} (see also~\cite[Thm.~5]{Frank-20_ppt}), we obtained the lower bounds mentioned in~\eqref{eq:values_p_c_d} for $p_c(d)$ in dimensions $d=1,2,3$. Should $c_{\rm LT}(3)$ be equal to its conjectured value $c_{\rm sc}(3)$, we would obtain the slightly better bound $p_c(3)\geq 1.494$. 

An exact upper bound on $I(d,p,1)$ can be obtained using Gaussian functions as trial states:
\begin{equation}
I(d,p,1)\leq -  \frac{d}{2}\left( 1 + \frac{2}{d} - p\right)
        \left( \frac{p-1}{\pi} \right)^{\frac{d(p-1)}{2 + d - dp}}
   p^{-\frac{2+d}{2 + d - dp}}.
\label{eq:upper_bd_I_Gaussian}
\end{equation}
This happens to be very precise in the regime of interest. In fact, using~\eqref{eq:upper_bd_I_Gaussian} and~\eqref{eq:lower_bd_LT}, we already obtain 
\begin{equation}
p_c(d)>\begin{cases}
  1.612& \text{for $d=1$,}\\
  1.526& \text{for $d=2$,}\\
  1.440& \text{for $d=3$.}
  \end{cases} 
 \label{eq:values_p_c_d_bis}
\end{equation}
We can therefore safely claim that $4/3<p_c(3)$. 
\end{remark}

\subsection{Proof of Theorem~\ref{thm:ground_states_orthonormal_all_p}}\label{sec:proof_existence_all_p}
We now prove that $J(N)$ has a minimiser for an infinity of integers $N \in \N$. Let us call $\cN\subset\N$ the set of all the integers $n$ which satisfies the binding inequalities $J(n)<J(k)+J(n-k)$ for all $k=1,...,n-1$. In particular, $J(n)$ has a minimiser for all $n\in\cN$. For $1<p<\min(2,1+2/d)$ we already know that $1,2\in\cN$. Our goal is to show that $\cN$ is not finite. 

If $N\notin\cN$, then there is $1 \le k \le N-1$ so that $ J(N) = J(k) + J(N-k)$. If $k \notin \cN$ or $N-k \notin \cN$, we can further decompose $J(k)$ or $J(N-k)$, and so on, until we obtain a decomposition of the form
\begin{equation}
    J(N)=\sum_{n\in \cN} k_n J(n),\quad N=\sum_{n \in \cN} k_n n, \quad k_n \in \N.
    \label{eq:decomp_not_in_cN}
\end{equation}
We claim that $k_n \in \{ 0, 1 \}$ and prove this by contradiction. Let $n_0 \in \cN$ be so that $k_{n_0} \ge 2$ in the previous decomposition. Then, since $n_0 \in \cN$, we have by Proposition~\ref{lem:strongBinding} that $J(2n_0)<2J(n_0)$. Together with the weak-binding inequality, this would imply
\[
    J(k_{n_0} n_0)\leq (k_{n_0}-2)J(n_0)+J(2n_0)<k_{n_0} J(n_0).
\]
In other words, if two of more "bubbles" have the same number of particles, it is energetically favourable for these bubbles to merge. This gives
\[
    J(N) = J \left( \sum_{n \in \cN \atop n \neq n_0} k_n n + k_{n_0} n_0 \right) <
    \sum_{n \in \cN \atop n \neq n_0} k_n J(n) + k_{n_0} J(n_0) = \sum_{n \in \cN} k_n J(N),
\]
which is in contradiction with~\eqref{eq:decomp_not_in_cN}. So the coefficients $k_n$ in~\eqref{eq:decomp_not_in_cN} must all be equal to $0$ or $1$. Now $\cN$ cannot be finite otherwise we would not be able to write all the $N\notin\cN$ as in~\eqref{eq:decomp_not_in_cN} with $k_n\in\{0,1\}$. This concludes the proof that $\cN$ is infinite. \qed

\subsection{Proof of Theorem~\ref{thm:limit_N} on the large--$N$ limit}\label{sec:proof_large-N-limit}

Using the Lieb-Thirring inequality~\eqref{eq:LT_orthonormal}, we have already seen in \textit{(i)} in Lemma~\ref{lem:weakBinding} that 
$$J(N)\geq e_{\rm LT}(d,p)\,N,\qquad \forall N\in\N.$$
In particular, we see that $J(N)/N$ is bounded from below. 

In Lemma~\ref{lem:weakBinding} we have also shown in \textit{(ii)} that $J(N)\leq J(N-K)+J(K)$ for every integer $K=1,...,N-1$. We deduce first that $J(N)/N \leq J(1) = I(1)<0$ which appears in the upper bound~\eqref{eq:upper_lower_bd_e}. The first inequality is in fact strict when $p<2$, by Corollary~\ref{cor:J(2N)}. Since the function $N\mapsto J(N)$ is subadditive and bounded from below by a constant times $N$ we conclude that $J(N)/N$ converges to its infimum $e(d,p)$ as in~\eqref{eq:limit_N_infty}, by Fekete's subadditive lemma~\cite{Fekete-23}.

It remains to prove the upper bound on $e(d,p)$ in terms of the semi-classical constant $c_{\rm sc}(d)$. Instead of using the Dirichlet eigenfunctions of a large domain $\Omega$, as we mentioned after Theorem~\ref{thm:limit_N},  we rather localise the periodic eigenfunctions in a cube (plane waves), which gives more explicit formulas. Let $C_L$ be the cube of side length $L$ centered at the origin and let $\chi\in C^\ii_c(\R^d,\R_+)$ be such that $\int_{\R^d}\chi=1$. Denote
$$u_{k}(x)=L^{-\frac{d}2}\sqrt{\1_{C_L}\ast\chi}\;e^{- i k\cdot x}$$
for $k\in (2\pi/L)\Z^d$. It turns out that these functions are orthonormal, since
$$\pscal{u_k,u_{k'}}=L^{-d}\int_{\R^d}\1_{C_L}\ast\chi\; e^{i(k-k')\cdot x}\,\rd x=(2\pi)^{\frac{d}2}\widehat{\1_{C_L}\ast\chi}(k-k')=0$$
due to the fact that $\widehat{1_{C_L}}(\ell)=0$ for $\ell\in (2\pi/L)\Z^d\setminus\{0\}$. 
A computation gives that the kinetic energy of each such function is equal to
$$\int_{\R^d}|\nabla u_k(x)|^2=|k|^2+\frac1{L^d}\int_{\R^d}|\nabla\sqrt{\1_{C_L}\ast\chi}(x)|^2\,\rd x.$$
The second term is a $O(1/L)$ since the function $\1_{C_L}\ast\chi$ is equal to 1 inside $C_L$, at a distance of order one to its boundary, and vanishes outside at a similar distance. We take $N \approx | C_L | \rho_* = L^d \rho_*$ such functions, each with a different $k$, with $\rho_*$ given by~\eqref{eq:formula_rho_*} in Lemma~\ref{lem:LTminimisation} and $C=c_{\rm sc}(d)$. We find the energy
$$\cE(u_{k_1},...,u_{k_N})=\sum_{j=1}^Nk_j^2-\frac{N^p}{pL^{dp}}\underbrace{\int_{\R^d}(\1_{C_L}\ast\chi)^p}_{=L^d+O(L^{d-1})}+O\left(\frac{N}{L}\right).$$
The first term is minimum when we take for the $k_j$ all the points of $(2\pi/L)\Z^d$ in a ball of fixed radius $R$, where $R$ is chosen so that there are $N$ points, that is, $R \approx \frac{1}{(2 \pi)} \left(\frac{d \rho_*}{| \SS^{d-1} |}\right)^{1/d}$.
Taking the limit $L\to\ii$ gives the upper bound in~\eqref{eq:upper_lower_bd_e}.\qed

\bigskip

Using that $J$ is Lipschitz by \textit{(iii)} in Lemma~\ref{lem:weakBinding}, we infer
\begin{align*}
\left|\frac{J(N+\alpha)}{N+\alpha} -\frac{J(N)}{N}\right|&\leq \frac{|J(N+\alpha)-J(N)|}{N+\alpha} +\frac{\alpha|J(N)|}{N(N+\alpha)}\\
&\leq \frac{\alpha}{N+\alpha}\big(C+e(d,p)\big)\underset{N\to\ii}{\longrightarrow}0,
\end{align*}
for any $\alpha\in[0,1]$. This proves that $J(\lambda)/\lambda$ has the same limit as when it is restricted to integers:
\begin{equation}
\boxed{\lim_{\lambda\to\ii}\frac{J(\lambda)}{\lambda}=e(d,p).}
\label{eq:limit_lambda}
\end{equation}
Using similar arguments as in the proof of Theorem~\ref{thm:ground_states_orthonormal_all_p}, we can then prove that $J(\lambda)/\lambda$ is always strictly above its limit, independently of whether it admits a minimiser or not.

\begin{corollary} \label{lem:decomposition_same_stable}
Let $d\geq1$ and $1 < p  < \min \{ 2, 1+\frac{2}{d} \}$. For all $\lambda>0$ and all $m\in\N\setminus\{1\}$ we have $J(m \lambda) < m J(\lambda)$. In particular, we have 
\begin{equation}
\frac{J(\lambda)}{\lambda}>e(d,p)
\end{equation}
for all $\lambda>0$.
\end{corollary}

\begin{proof}
    If $J(\lambda)$ has a minimiser, this was already proved in Corollary~\ref{cor:J(2N)}. This covers in particular the case $0\leq\lambda\leq1$. If $J(\lambda)$ does not have a minimiser, then, according to Remark~\ref{rem:binding_noninteger}, there is an integer $1 \le k \le \lambda$ so that $J(\lambda)= J(k) + J(\lambda - k)$. By further decomposing $J(k)$, we can therefore write as in~\eqref{eq:decomp_not_in_cN},
    \[
         J(\lambda)=\sum_{n\in \cN} k_n J(n) + J(\lambda - k), \quad \lambda =\sum_{n \in \cN} k_n n + (\lambda - k), \quad k_n \in \{0, 1\},
    \]
    and at least one $k_n$ has value $1$. So we have, as before
    \[
        mJ(\lambda) = \sum_{n\in \cN} m k_n J(n) + m J(\lambda - k) >
        \sum_{n\in \cN} k_n J(m n) + J(m(\lambda - k)) \ge J(m\lambda).
    \]
    This concludes the proof of Corollary~\ref{lem:decomposition_same_stable}.
\end{proof}

\section{Application: Symmetry breaking for a crystal in the Kohn-Sham model with large Dirac exchange}
\label{sec:rHFD}

In this section we explain how the previous results can be used to prove symmetry breaking for an infinite periodic system, within a simple Kohn-Sham model with a Dirac (a.k.a.~Slater) term. The results of this section are similar to a recent work by Ricaud~\cite{Ricaud-17} on the Thomas-Fermi-von Weis\"acker-Dirac model, so some technical details will be omitted for shortness. The main difference is that we deal with operators instead of functions. Our results can be generalised to other contexts, such as the symmetry breaking in the dissociation of the hydrogen molecule considered in~\cite{Hol_etal-19_ppt}.

\subsection{Notation and main results}
Everywhere in this section we fix the dimension $d = 3$. Let $\cR$ be a lattice of $\R^3$, with unit cell denoted by $\WS$ and dual lattice by $\cR^*$. We consider the infinite system obtained by placing one point nucleus of charge $Z=N$ at each site of this lattice,\footnote{More generally we could place several nuclei of charges $z_1,...,z_M$ in each unit cell, so that the total charge is $\sum_{m=1}^Mz_m=N$.} together with an infinite sea of quantum electrons in a periodic state. We assume that the system is locally neutral, which means that the number of electrons per unit volume must be equal to $N|\WS|^ {-1}$. Our goal is to determine whether these electrons will have the same periodicity $\cR$ as the lattice of the nuclei or whether it is more favourable energetically to place them with a different period. In the latter case we say that there is \emph{spatial symmetry breaking}. More specifically, we will study whether the $(\ell\cR)$-periodic electronic ground state is $\cR$-periodic or not, for $\ell\geq2$. 

We recall that an \emph{$\cR$-periodic density matrix} $\gamma$ is a self-adjoint operator $0\leq \gamma=\gamma^*\leq1$ on $L^2(\R^3)$ (we neglect the spin for simplicity) which commutes with all the translations of the lattice $\cR$:
$$\forall R \in \Lat,\qquad \tau_R \gamma = \gamma \tau_R.$$
Here $\tau_R$ is the unitary operator on $L^2(\R^3)$ defined by $(\tau_R f)(x) := f(x - R)$. We restrict ourselves to density matrices which have a finite trace and a finite kinetic energy per unit volume, which means that $\gamma$ and $\sqrt{-\Delta}\gamma\sqrt{-\Delta}$ are locally trace-class. The density of $\gamma$ is the unique $\Lat$-periodic function $\rho_\gamma\in L^1_{\rm loc}(\R^3,\R_+)$ such that 
$$\tr(\chi\gamma\chi)=\int_{\R^3}\chi(x)^2\rho_\gamma(x)\,\rd x$$
for every $\chi\in L^\ii(\R^3)$ of compact support. Any such density matrix $\gamma$ represents an infinite periodic system of electrons. The number of electrons in each unit cell is defined by
\[
    \VTr_{\Lat}(\gamma) :=\Tr \left(  \1_{\WS} \gamma \1_{\WS} \right) = \int_{\WS} \rho_\gamma(x)\,\rd x.
\]

In this section we work with electronic density matrices which are $(\ell\cR)$-periodic for some $\ell\geq1$. All the previous definitions are easily extended to the case $\ell\geq2$. Our main goal is to determine whether an $(\ell\cR)$-periodic minimiser is necessarily $\cR$-periodic or not. In what follows, we enforce neutrality of the system. So, in any supercell of the type $\ell \WS$, we impose
\[
   \boxed{ \VTr_{\ell \Lat} (\gamma) =\int_{\ell\WS}\rho_\gamma= \ell^3 N. }
\]
The functional to minimise is the Kohn-Sham energy per unit cell which is defined by
\begin{multline} \label{eq:def:EL}
 \cE_{c,\ell}^{\rm KS}(\gamma) := \VTr_{\ell \Lat} \left( - \Delta \gamma  \right) - N \int_{\ell \WS}  G_{\Lat}(x) \rho_\gamma(x)\,\rd x + \frac1{2} D_{\ell \Lat}(\rho_\gamma, \rho_\gamma)\\ - \frac{3c}{4} \int_{\ell \WS} \rho_\gamma(x)^{\frac43}\,\rd x
\end{multline}
for any $(\ell\cR)$-periodic density matrix $\gamma$. 
The first term is the kinetic energy per unit cell $\ell \WS$, interpreted in the sense of quadratic forms. The second term is the interaction between the $(\ell\cR)$-periodic electrons and the lattice $\cR$ of the nuclei of charge $Z=N$. The function $G_\cR$ is the $\Lat$-periodic Green's function, solution to the periodic Laplace equation
\begin{equation}
 -\Delta G_\Lat = 4 \pi \left(\sum_{R \in \Lat} \delta_R -  | \WS |^{-1}\right)  \quad \text{and} \quad \int_\WS G_\Lat = 0.
 \label{eq:periodic_Coulomb}
\end{equation}
In other words, $NG_\cR$ is the Coulomb potential of the infinite lattice of nuclei, screened by a uniform background. The third term in~\eqref{eq:def:EL} is the Coulomb interaction between the electrons in the Hartree approximation and it reads
\[
D_{\ell\Lat}(f,g) := \int_{\ell\WS}\int_{\ell\WS} G_{\ell\Lat}(x - y) f(x) g(y) \rd x\, \rd y
\]
where $G_{\ell\cR}$ is defined similarly as in~\eqref{eq:periodic_Coulomb} with $\cR$ replaced by $\ell\cR$ and $\WS$ replaced by $\ell\WS$. 
Finally, the last term of~\eqref{eq:def:EL} is the Dirac or Slater term. This term is an approximation of the exchange-correlation energy of $\gamma$, in terms of the density $\rho_\gamma$ only. The parameter $c \ge 0$ usually has a fixed value given by physical considerations (for the exchange part of the energy the constant is $c\simeq1.24$ without spin~\cite[Sec.~6.2]{LieSei-09}). Here, we change the value of $c$, and compare the resulting energies for different values of the periodicity $\ell$ of the electrons.  

The minimisation problem to be considered reads 
\begin{equation} \label{eq:def:IL}
\boxed{E^{\rm KS}(c,\ell) := \min \Big\{ \cE^{\rm KS}_{c,\ell}(\gamma), \  \gamma\  \text{$(\ell\cR)$-periodic density matrix, }  \ \int_{\ell\WS}\rho_\gamma = \ell^3N   \Big\}.}
\end{equation}
The existence of minimisers easily follows from the direct method of the calculus of variations, since the problem is posed on the compact set $\ell\WS$~\cite{CatBriLio-01,CanDelLew-08a}. 
An $\cR$-periodic state is of course $(\ell\cR)$-periodic and its Kohn-Sham energy is found to be equal to $\cE^{\rm KS}_{c,\ell}(\gamma)=\ell^3\cE^{\rm KS}_{c,1}(\gamma)$. In particular we deduce that 
$$E^{\rm KS}(c,\ell)\leq \ell^3E^{\rm KS}(c,1)$$
for every $\ell\in\N$ and every $c\geq 0$. 

\begin{definition}[Symmetry breaking]
    We say that there is \emph{spatial symmetry breaking} for the Dirac-Kohn-Sham model with parameter $c \ge 0$ if there exists $\ell \in\N\setminus\{1\}$ such that
    \[
    \boxed{\frac{E^{\rm KS}(c,\ell)}{\ell^3} < E^{\rm KS}(c,1) . }
    \]
\end{definition}

The definition means that an $\ell$-periodic minimiser has a lower energy per unit volume than the $1$-periodic state. It does not mean that the electrons will necessarily be in this $\ell$-periodic state. But at least we can deduce that they will not be $1$-periodic. 

The case $c=0$ is studied at length in~\cite{CatBriLio-01} and in~\cite[App.~A]{CanDelLew-08a}. In this situation the energy $\gamma\mapsto \cE^{\rm KS}_{0,\ell}(\gamma)$ is convex and the problem $E^{\rm KS}(0,\ell)$ admits a unique minimiser $\gamma_{0,\ell}$ for every $\ell\geq1$. This state solves the nonlinear equation
\begin{equation}
\gamma_{0,\ell}=\1\left(-\Delta - N G_\cR + \rho_{\gamma_{0,\ell}} \ast G_{\ell\cR} \leq \eps_{0,\ell}\right)
\label{eq:SCF_periodic}
\end{equation}
where $\eps_{0,\ell}$ is a Lagrange multiplier chosen to enforce the constraint that $\gamma_{0,\ell}$ has $N\ell^3$ electrons per unit cell. In addition, it is unique in the sense that any $(\ell\cR)$-periodic solution to equation~\eqref{eq:SCF_periodic} for some $\eps_{0,\ell}$ with the right number of electrons $N\ell^3$ must be equal to $\gamma_\ell$. Since the $\cR$-periodic state $\gamma_{0,1}$ with $\ell=1$ is a solution for all $\ell\geq2$, it follows that 
$$\gamma_{0,\ell}=\gamma_{0,1}\text{ and } \eps_{0,\ell}=\eps_{0,1},\qquad\text{for all $\ell\in\N$, when $c=0$}$$
and therefore that
$$E^{\rm KS}(0,\ell)=\ell^3E^{\rm KS}(0,1),\qquad\text{for all $\ell\in\N$, when $c=0$}.$$
No symmetry breaking occurs for $c=0$. For latter purposes, we mention that the system is called an \emph{insulator} when $\eps_{0,1}$ can be chosen in a spectral gap of the operator $-\Delta - N G_\cR + \rho_{\gamma_{0,1}} \ast G_{\cR}$ in~\eqref{eq:SCF_periodic} and that it is a \emph{metal} otherwise. Which of the two cases occurs depends on the shape of the lattice $\cR$ and on the number of particles $N$ per unit cell. 

The Dirac term is not convex when $c>0$. It is natural to expect that symmetry will not be broken for $c$ small enough whereas it could be broken for large $c$. This is confirmed by the following result.

\begin{theorem}[Occurrence of symmetry breaking] \label{th:main_periodic}
Let $\cR$ be a lattice in $\R^3$ and $N\in\N$. There is a critical $c^*=c^*(\cR,N)\in[0,\infty)$ such that, for all $c > c^*$, the system breaks spatial symmetry. 
In addition, if the system is insulating at $c = 0$, then $c^* > 0$: there exists $c_0^*>0$  such that $E^{{\rm KS}}(\ell, c) = \ell^3E^{\rm KS}(1, c)$ for all $\ell \in \N$ and all $c<c_0^*$ .
\end{theorem}

The spirit of the result is exactly the same as~\cite{Ricaud-17} in the Thomas-Fermi-von Weis\"acker-Dirac case. For $c$ very large the kinetic energy and the Dirac term dominate, the other terms being of lower order. The very large constant $c$ has the effect of concentrating the electrons at the scale $1/c$. After rescaling length by a factor $1/c$ about a blow-up point, in the limit the problem converges to the fermionic NLS problem in the whole space with $p=4/3$. This is the content of the following result.  

\begin{proposition}[Convergence to the NLS problem in $\R^3$] \label{lem:periodic_and_limiting}
Let $\cR$ be a lattice in $\R^3$ and $N\in\N$. For all $\ell \in \N$, we have
    \[
    \lim_{c \to \infty} \frac{E^{{\rm KS}}(c,\ell)}{c^2}  = J( \ell^3 N)
    \]
   where $J(\ell^3N)$ is defined as in~\eqref{eq:def_J_integers} with $d = 3$ and $p = 4/3$.
\end{proposition}

In Theorem~\ref{thm:ground_states_orthonormal_all_p} and in Corollary~\ref{lem:decomposition_same_stable} we have proved that $J(\ell^3N)<\ell^3J(N)$ for all $\ell\in\N\setminus\{1\}$. This shows that for $c$ large enough $E^{{\rm KS}}(c,\ell)<\ell^3E^{{\rm KS}}(c,1)$, hence that there is symmetry breaking. The intuitive picture is that it is more favourable to concentrate $\ell^3N$ particles at one point rather than having $\ell^3$ bumps of $N$ concentrated electrons, as is the case for the $\cR$-periodic minimiser placed in the $(\ell\cR)$-periodic energy.  This is how we can prove the first part of Theorem~\ref{th:main_periodic}.

\begin{remark}
    The previous result does not use that $p=4/3 < p_c(3)$. However, since this inequality has been numerically found to hold (see Remark~\ref{rmk:comput_p_c}), minimisers for $J(N)$ always exist, and one can say more. Following the approach of~\cite{Ricaud-17}, it is possible to prove that minimisers $\gamma_c$ for $E^{\rm KS}(c,1)$ satisfy $U_c\gamma_c U_{c^{-1}}\rightharpoonup \gamma$ weakly-$\ast$ locally in the trace class, where $\gamma$ minimises $J(N)$ and $U_c$ is the dilation operator defined by $(U_cf)(x)=c^{3/2}f(cx)$. In other words, the electrons concentrate at the origin where the nucleus is placed, in the unit cell $\WS$. For $E^{\rm KS}(c,\ell)$ the result is similar but the $\ell^3N$ electrons concentrate at one of the $\ell^3$ nuclei of the larger unit cell $\ell\WS$. Finally, we have the expansion
\begin{multline*}
E^{\rm KS}(c,\ell)= J( \ell^3 N)c^2\\
+c\min_{\substack{\gamma\ {\rm min.}\\ {\rm for }\ J(\ell^3N)}}\left(-\int_{\R^3}\frac{\rho_\gamma(x)}{|x|}\,\rd x+\frac12\iint_{\R^3\times\R^3}\frac{\rho_\gamma(x)\rho_\gamma(y)}{|x-y|}\,\rd x\,\rd y\right)+o(c).
\end{multline*}
\end{remark}

The literature contains several results in the same spirit as Theorem~\ref{th:main_periodic} and Proposition~\ref{lem:periodic_and_limiting}. The closest to our work is~\cite{Hol_etal-19_ppt} which studies the case of the Kohn-Sham hydrogen molecule (two electrons in the field of two nuclei separated by a distance $R$). This corresponds to $N=2$ but since the spin is taken into account there is no orthogonality constraint between $u_1$ and $u_2$. In this model, spin symmetry breaking arises in the limit $c\to\ii$ because each electron has to concentrate about one of the two nuclei, where it asymptotically solves the NLS problem $I(1)$. This is therefore a completely different phenomenon from this present work, where the two particles concentrate at the \emph{same} point. Other works in the same spirit include for instance~\cite{AshFroGraSchTro-02,GuoSei-14,GouZenZho-16} for the \emph{Hartree model} in multiple well potentials, which also has no orthogonality constraint. 

In the next section we outline the proof of Proposition~\ref{lem:periodic_and_limiting} whereas in Section~\ref{sec:small_c} we quickly discuss the absence of symmetry breaking for $c$ small enough, under the additional assumption that the system is an insulator at $c=0$ (second part of Theorem~\ref{th:main_periodic}).

\subsection{Sketch of the proof of Proposition~\ref{lem:periodic_and_limiting} and of the first part of Theorem~\ref{th:main_periodic}}
The symmetry breaking stated in the first part of Theorem~\ref{th:main_periodic} follows immediately from Proposition~\ref{lem:periodic_and_limiting}, the proof of which we outline in this section.

We set for simplicity $\ell = 1$ (the proof is similar in the general case). For $\check{\gamma}$ an $\Lat$-periodic density matrix, we call $\gamma_c=U_c\check\gamma U_{c^{-1}}$ the rescaled operator whose kernel is
\begin{equation} \label{eq:def:scaling}
\gamma_c(x, y) := c^{-3} \check{\gamma} ( x/c,y/c).
\end{equation}
Using that $G_{c \Lat}(x) = c^{-1} G_\Lat(c^{-1} x)$, we obtain the following scaling relations:
$$ \VTr_{c \Lat}( \gamma_c ) = \VTr_{\Lat}( \check{ \gamma} ), \qquad \int_{c \WS} G_{c \Lat} \rho_c  = \frac{1}{c} \int_{\WS} G_\Lat \check{\rho},\qquad D_{c \Lat}(\rho_c, \rho_c)  = \frac{1}{c} D_\Lat(\check{\rho}, \check{\rho}),$$
$$\VTr_{c \Lat}( - \Delta \gamma_c)  = \frac{1}{c^2} \VTr_{\Lat}( - \Delta \check{ \gamma} ),\qquad \int_{ c \WS} \rho_c^{4/3}  = \frac{1}{c^2} \left( c \int_{\WS} \check{\rho}^{4/3} \right).$$
We deduce that the energy of $\check{\gamma}$ can be re-expressed as
\begin{equation} \label{eq:def:Erescaled}
\cE^{\rm KS}_{c,1}(\check{\gamma})  := c^2 \cE_{c \Lat}( \gamma_c ) + c \cF_{c \Lat}(\gamma_c),
\end{equation}
with
\[
\cE_{c \Lat}(\gamma_c) := \VTr_{c \Lat} ( - \Delta \gamma_c) - \frac{3}{4}\int_{c \WS} \rho_c^{4/3}, \qquad
\cF_{c \Lat}(\gamma_c) := -N\int_{c \WS} G_{c \Lat} \rho_c + D_{c \WS}(\rho_c, \rho_c).
\]
The energy $\cE_{c \Lat}$ is similar to the NLS energy $\cE$ in~\eqref{eq:def_energy_gamma} except that the problem is restricted to the flat torus of size $c$, instead of being posed over the whole of $\R^3$. 

\subsubsection*{\bf Step 1.} Let us first prove that 
\begin{equation}
\limsup_{c \to \infty} \frac{E^{{\rm KS}}(c,1)}{c^2} \le  J(N) 
\label{eq:limsup}
\end{equation}
Let $\gamma$ be a smooth rank-$N$ projector of compact support such that\footnote{To obtain such a $\gamma$ one can start with a trial state $\gamma=\sum_{i=1}^N|u_i\rangle\langle u_i|$ and then truncate and regularise the $u_i$'s. The new functions can be orthonormalised using the same procedure as in Lemma~\ref{lem:weakBinding}.} $\cE(\gamma)\leq J(N)+\eps$.
This state can be used as a trial state in the rescaled box $c\WS$, as soon as its support is strictly included in $c\WS$. This amounts to $(c\cR)$-periodising $\gamma$ in the manner $\sum_{ R \in \Lat} \tau_{c R}^* \gamma  \tau_{cR}$.
Then $\cE_{c \Lat}(\gamma)=\cE(\gamma)$ whereas 
$$\lim_{c\to\ii}\cF_{c \Lat}(\gamma)=-N\int_{\R^3}\frac{\rho_\gamma(x)}{|x|}\,\rd x+\frac12 \iint_{\R^3\times\R^3}\frac{\rho_\gamma(x)\rho_\gamma(y)}{|x-y|}\rd x\,\rd y.$$
Hence  
$$E^{\rm KS}(c,1)\leq c^2(J(N)+\eps)+O(c).$$
The claimed bound~\eqref{eq:limsup} follows after taking $c\to\ii$ and then $\eps\to0$.

\subsubsection*{\bf Step 2.} To prove the other inequality 
\begin{equation}
\liminf_{c \to \infty} \frac{E^{{\rm KS}}(c,1)}{c^2} \ge  J(N) 
\label{eq:liminf}
\end{equation}
we consider a minimiser $\check\gamma_c$ for $E^{{\rm KS}}(c,1)$ and call $\gamma_c$ the rescaled operator as in~\eqref{eq:def:scaling}. From the previous step and the positivity of the Hartree term, we have for $c$ large enough
    \begin{equation} \label{eq:ineq_for_control_kinetic}
        \VTr_\Lat( - \Delta \check{\gamma_c}) - N \int_{\WS} G_\Lat \check{\rho_c} - \frac{3c}{4} \int_{\WS} \check{\rho_c}^{4/3} \le c^2\frac{J(N)}2.
    \end{equation}
    Using the Gagliardo-Nirenberg and Hoffmann-Ostenhof periodic inequalities, we have
    \[
        \int_{\WS}\check{\rho_c}^{\frac43}  \le  C_1 N^{\frac56} \left(\int_{\WS}\rho_c+ \int_{\WS} | \nabla  \sqrt{\rho_c} |^2 \right)^{\frac12} \le  C_2 N^{\frac56} \Big(N+\VTr_\Lat( - \Delta \check{\gamma_c})\Big)^{\frac12}.
    \]
   Similarly, to control the potential energy, we use that $G_{\Lat}\leq |x|^{-1}+C$, and obtain by Hardy's inequality 
    $$\int_{\WS} G_\Lat \check{\rho_c} \leq C_2 \sqrt{N} \Big(N+\VTr_\Lat( - \Delta \check{\gamma_c})\Big)^{\frac12}.$$
    Inserting in~\eqref{eq:ineq_for_control_kinetic} this gives $\VTr( - \Delta \check{\gamma_c})=O(c^2)$ and hence after scaling we obtain
   $$\VTr( - \Delta \gamma_c)=O(1),\qquad \int_{c\WS}G_{c \Lat} \rho_c=O(1).$$
   This gives
  $$ E^{\rm KS }(c,1)\geq c^2\cE_{c \Lat}(\gamma_c)+O(c).$$
    
The last step is to show that 
\begin{equation}
 \liminf_{c\to\ii}\cE_{c \Lat}(\gamma_c)\geq J(N).
 \label{eq:liminf_final}
\end{equation}
To prove~\eqref{eq:liminf_final} we decompose $\gamma_c$ into bubbles. We use the operator version of the bubble decomposition, which has implicitly appeared several times in the literature and can be read with full details in the recent work~\cite[Theorem 3.1]{HonKwoYoo-19}. The present setting is slightly different from~\cite{HonKwoYoo-19} due to the periodic boundary condition but the proof is similar, see, e.g.,~\cite{Ricaud-17} in the case of functions. For operators the result is that there exists a sequence of density matrices $\{ \gamma^{(1)}, \gamma^{(2)}, \cdots \}$ over $\R^3$ with $\Tr( - \Delta \gamma^{(i)}) < \infty$ such that 
$$N\geq \sum_i \Tr(\gamma^{(i)}),$$
$$\liminf_{c\to\ii}\VTr_{c \Lat} ( - \Delta \gamma_c) \geq \sum_{i}\Tr(-\Delta \gamma^{(i)})$$
and
$$\lim_{c\to\ii}\int_{c\WS}\rho_{\gamma_c}^{\frac43}=\sum_{i}\int_{\R^3}\rho_{\gamma^{(i)}}^{\frac43}.$$
The Dirac term decomposes exactly since $4/3$ is a sub-critical power, whereas for the mass and the kinetic energy one only obtains lower bounds. The missing mass and kinetic energy are contained in the \emph{vanishing} part of $\gamma_c$, to employ the vocabulary of the concentration-compactness method. Each $\gamma^{(i)}$ is constructed as the strong local limit of $\chi_{i,c}(\cdot+x_{i,c})\gamma_c\chi_{i,c}(\cdot+x_{i,c})$ for some translation $x_{i,c}$ and some localisation function $\chi_{i,c}$, with $|x_{i,c}-x_{j,c}|\to\ii$ when $i\neq j$, up to subsequences. Using the subadditivity of $J$ proved in Lemma~\ref{lem:weakBinding}, we deduce that 
\begin{equation}
\liminf_{c\to\ii}\cE_{c\cR}(\gamma_c)\geq \sum_{i}\cE(\gamma^{(i)})\geq \sum_i J\big(\Tr(\gamma^{(i)})\big)\geq J\left(\sum_i\Tr(\gamma^{(i)})\right)\geq J(N).
\label{eq:decomp_bubble_energy}
\end{equation}
This concludes our sketch of the proof of Proposition~\ref{lem:periodic_and_limiting}.\qed


\subsection{Proof of the second point of Theorem~\ref{th:main_periodic}: stability for small $c$}\label{sec:small_c}
When $c = 0$, we have recalled from~\cite[App.~A]{CanDelLew-08a} that the minimisation problem $E^{\rm KS}(0,1)$ admits a unique minimiser, which we denote here by $\gamma_0$ (it was called $\gamma_{0,1}$ above). It solves the nonlinear operator equation
\begin{equation}
     \gamma_0 = \1(H_0 \le \varepsilon_0)
 \label{eq:H_0}
\end{equation}
where $\eps_0$ is a Lagrange multiplier and 
\[
    H_{0} := - \Delta - N G_{\Lat}  + \rho_{0} \ast G_{\Lat}.
\]
The assumption that the system is an insulator means that $\eps_0$ belongs to a spectral gap of the operator $H_0$. For simplicity, we denote by
$$a:=\max\sigma(H_0)\cap (-\ii,\eps_0), \qquad b:=\min\sigma(H_0)\cap (\eps_0,\ii)$$
and, without loss of generality, we can choose 
\[
    \eps_0=\frac12 (a+b).
\]
The length of the gap is $g:=b-a>0$. Let us prove that $\rho_0>0$. After a Bloch-Floquet transform~\cite{ReeSim4}, the equation~\eqref{eq:H_0} means that the kernel of $\gamma_0$ is given by
$$\gamma_0(x,y)=\sum_{n\geq0}\int_B \1(\lambda_n(\xi)\leq\eps_0)\, u_n(\xi,x)\overline{u_n(\xi,y)}\,\rd \xi$$
with the density 
$$\rho_0(x)=\sum_{n\geq0}\int_B \1(\lambda_n(\xi)\leq\eps_0)\, |u_n(\xi,x)|^2\rd \xi.$$
Here $B$ is the Brillouin zone (the unit cell of the dual lattice $\cR^*$) and $(u_n,\lambda_n)$ are the Bloch eigenfunctions and (ordered) eigenvalues, which solve
$$\left(|-i\nabla_x+\xi|^2- N G_{\Lat}  + \rho_{0} \ast G_{\Lat}\right)u_n(\xi,\cdot)=\lambda_n(\xi)\,u_n(\xi,\cdot)$$
with periodic boundary conditions on $\partial\WS$. By Perron-Frobenius we have $u_0>0$ and $\lambda_0(0)<\lambda_1(0)$. By perturbation theory we then deduce that $\lambda_0(\xi)$ is non-degenerate with a positive eigenfunction $u_0(\xi,\cdot)$, for $\xi$ small enough. Then $\rho_0>0$ and in the following we denote by 
\[
\alpha := \min_{\WS} \rho_{0} > 0
\]
the minimal value of the periodic density. 

The following shows that the gap does not close and the density stays strictly positive for $c$ small enough.

\begin{lemma}[Stability of the gap]\label{lem:stability_gap}
There is $c_1 > 0$ such that, for all $0 \le c < c_1$, any minimiser $\gamma_c$ for $E^{\rm KS}(c,1)$ satisfies 
    \[
         \min_{\WS} \rho_{c} > \frac{\alpha}{2}, \quad \text{and} \quad
         {\rm dist} \left( \sigma( H_{c} ) , \varepsilon_0 \right) > \frac{g}{4},
    \]
    where we set $\rho_c := \rho_{\gamma_c}$, and
    \[
        H_c := - \Delta - N G_{\Lat} + \rho_c * G_\Lat - c \rho_c^{1/3}.
    \]
    Finally, we have $\gamma_c = \1 (H_c < \varepsilon_0)$ and there is $C > 0$ independent of $c<c_1$ so that the following operator inequality holds:
    \begin{equation} \label{eq:H-epsF}
        C^{-1} (1 - \Delta)\leq | H_{c} - \varepsilon_F | \leq C (1 - \Delta).
    \end{equation}
\end{lemma}

\begin{proof}
Let $c_n\to0^+$. The energy $c\mapsto E^{\rm KS}(c,1)$ is continuous at $c=0$ and any minimiser $\gamma_{c_n}$ is a minimising sequence for $E^{\rm KS}(0,1)$. Hence it must converge to the unique minimiser $\gamma_0$ weakly and 
$$\lim_{c_n\to0^+}\VTr_{\Lat} ( - \Delta \gamma_{c_n})=\VTr_{\Lat} ( - \Delta \gamma_0).$$
This implies that $\rho_{c_n}\to\rho_0$ strongly in $L^1\cap L^3(\WS)$. Then we write the associated mean-field operator in the form
$$H_{c_n} = - \Delta - N G_{\Lat} + \rho_{c_n} * G_\Lat - c_n \rho_{c_n}^{1/3}=H_0+(\rho_{c_n}-\rho_0) * G_\Lat - c_n \rho_{c_n}^{1/3}$$
and estimate the operator norms of the last two terms by 
$$\norm{(\rho_{c_n}-\rho_0) * G_\Lat(1-\Delta)^{-1}}\leq \norm{\rho_{c_n}-\rho_0}_{L^1(\WS)}\norm{G_\Lat(1-\Delta)^{-1}}\to0$$
and
$$\norm{\rho_{c_n}^{1/3}(1-\Delta)^{-1}}\leq C\norm{\rho_{c_n}}_{L^1(\WS)}^{1/3}.$$
With similar estimates we know that $(H_0+C)(1-\Delta)^{-1}$ and $(H_0+C)^{-1}(1-\Delta)$ are bounded for $C$ large enough, locally uniformly in $c$, see~\cite[Lem.~1]{CanDelLew-08a}. By the Rellich-Kato theorem, this proves that the spectrum of $H_{c_n}$ converges to that of $H_0$. In particular, $H_{c_n}$ has a gap around $\eps_0$, independent of $c_n$ for $c_n$ small enough. To conclude we have therefore shown that there exists $c_1>0$ so that any minimiser $\gamma_c$ for $E^{\rm KS}(c,1)$ has a mean-field operator $H_c$ with the gap $g/4$ around $\eps_0$. This implies~\eqref{eq:H-epsF} by~\cite[Lem.~3]{CanDelLew-08a}.

Let then $\gamma_c$ be any such minimiser for $c<c_1$. Since the family $H(t)=H_0+t(\rho_{c}-\rho_0) * G_\Lat - tc \rho_{c}^{1/3}$ has a gap for all $t\in[0,1]$ and the rank of a continuous family of orthogonal projectors is always constant, we obtain
$$\VTr_\cR\1(H_{c}\leq\eps_0)=N.$$
By~\cite[App.~A]{CanDelLew-08a} we know that $\gamma_{c}=\1(H_{c}\leq\eps_c)$ where $\eps_c$ is the unique Lagrange multiplier chosen such that $\VTr(\gamma_c)=N$, and we conclude that $\eps_c=\eps_0$ is independent of $c$. In particular $\gamma_c=\1(H_{c}\leq\eps_0)$.

Finally, we have 
$$\VTr(C+H_c)\gamma_c(C+H_c)\leq (C+\eps_0)^2N.$$
Since $(H_0+C)(1-\Delta)^{-1}$ and $(H_0+C)^{-1}(1-\Delta)$ are bounded, this shows that $\VTr(1-\Delta)\gamma_c(1-\Delta)\leq C$ uniformly in $c<c_1$. This implies that $\rho_c$ is bounded in $W^{2,1}(\WS)$ and therefore we have $\rho_c\to\rho_0$ in $L^\ii(\WS)$ when $c\to0^+$. In particular $\rho_c\geq \alpha/2>0$ for $c$ small enough. 
\end{proof}

Next we use the properties of minimisers for $c<c_1$ in Lemma~\ref{lem:stability_gap} to show that there is indeed only one, for every $\ell\geq1$. 

\begin{lemma}
There is $c_2 > 0$ so that, for all $0 \le c < c_2$, $E^{\rm KS}(c,1)$ has a unique minimiser $\gamma_c$, satisfying the properties of Lemma~\ref{lem:stability_gap}. This minimiser is also the unique minimiser for $E^{\rm KS}(c,\ell)$ for all $\ell\geq1$, hence there is no symmetry breaking for $c<c_2$. 
\end{lemma}

\begin{proof}
We use the framework developed in~\cite{BacBarHelSie-99,HaiLewSer-05a,CanDelLew-08a, CanDelLew-08b,FraLewLieSei-12}. Let $\gamma_c$ be any minimiser for $E^{\rm KS}(c,1)$ with $c<c_1$ and let $\gamma$ be any other $(\ell\cR)$-periodic density matrix. Using $\VTr_{\ell\cR}\gamma=\VTr_{\ell\cR}\gamma_c=\ell^3N$, we can rewrite and estimate the difference of the two energies as
\begin{align}
&\cE^{\rm KS}_{c,\ell}(\gamma)-\cE^{\rm KS}_{c,\ell}(\gamma_c)\nn\\
&\qquad=\VTr_{\ell \cR} (H_{c} - \varepsilon_0)Q + \frac12 D_{\ell \Lat}(\rho_{Q}, \rho_{Q})- \frac34 c \int_{\ell \WS} \left( (\rho_c + \rho_{Q})^{\frac43} - \rho_{c}^{\frac43} - \frac43 \rho_{c}^{\frac13} \rho_{Q}   \right)\nn \\
&\qquad\geq \VTr_{\ell \cR} (H_{c} - \varepsilon_0)Q + \frac12 D_{\ell \Lat}(\rho_{Q}, \rho_{Q}) -cK \int_{\ell \WS} \min\left(\rho_{Q}^2\,,\, \rho_{Q}^{\frac43}\right)\label{eq:def:FQc}
\end{align}
where $Q:=\gamma-\gamma_c$. In the second line we have used that 
$$(1 + t)^{4/3} - 1 - \frac43 t \le C \min(t^{4/3},t^2)$$ 
for all $t \ge -1$ and that $\rho_c\geq \alpha/2$. 
Our goal is to show that~\eqref{eq:def:FQc} is non-negative and vanishes only at $\gamma=\gamma_c$. We claim that 
\begin{equation}
\int_{\ell \WS} \min\left(\rho_{Q}^2\,,\, \rho_{Q}^{\frac43}\right)\leq C\VTr_{\ell \cR} (H_{c} - \varepsilon_0)Q.
\label{eq:estim_to_be_proven_for_Q}
\end{equation}
The result then follows under the assumption that $c<c_2 := \min(c_1,(2CK)^{-1})$. 

To prove~\eqref{eq:estim_to_be_proven_for_Q} we introduce
\begin{align*}
    Q^{--}  := \gamma_{c} Q \gamma_{c}, \qquad
    &  Q^{-+}  := \gamma_{c} Q (1 - \gamma_{c}), \\
    Q^{+-} := (1 - \gamma_{c}) Q \gamma_{c}, \qquad
    & Q^{++} :=  (1 - \gamma_{c}) Q (1 - \gamma_{c})
\end{align*}
and note that 
\begin{multline*}
\VTr_{\ell \cR} (H_{c} - \varepsilon_0)Q=\VTr_{\ell \cR} |H_{c} - \varepsilon_0|(Q^{++}-Q^{--})\\
\geq C\VTr_{\ell \cR} (1-\Delta)(Q^{++}-Q^{--})\geq C\VTr_{\ell \cR} (1-\Delta)Q^2. 
\end{multline*}
We have used Bach's inequality $Q^2\leq Q^{++}-Q^{--}$ from~\cite[Eq. (18)--(19)]{BacBarHelSie-99}. 

For $q=Q^{++},Q^{--}$ we use the Lieb-Thirring inequality which implies 
$$\VTr_{\ell\cR}(1-\Delta)q\geq C\int_{\ell\WS}\rho_q+\rho_q^{\frac53}\geq 2C\int_{\ell\WS}\rho_q^{\frac43}$$
and provides the desired bound on the two densities $\rho_{Q^{++}}$ and $\rho_{Q^{--}}$.

For $Q^{+-}$ and $Q^{-+}$ the argument is slightly more involved. Following~\cite[Prop.~1]{CanDelLew-08a} we claim that 
\begin{equation}
 \int_{\ell\WS}\rho_{Q^{+-}}^2+\rho_{Q^{-+}}^2\leq C\,\VTr_{\ell \cR} Q^2
 \label{eq:estim_Q_+-}
\end{equation}
where the constant $C$ is independent of $\ell$. The argument goes by duality in the form
\begin{align*}
\left|\int_{\ell\WS}\rho_{Q^{+-}}V\right| & =\left|\VTr_{\ell\cR}\Big(\gamma_cV(1-\gamma_c)Q\Big)\right|\leq \norm{Q}_{\gS^2(L^2(\ell\WS))}\norm{\gamma_c V}_{\gS^2(L^2(\ell\WS))} \\
& \le \norm{Q}_{\gS^2(L^2(\ell\WS))} \norm{\gamma_c (1 -\Delta )} \norm{(1 - \Delta)^{-1} V}_{\gS^2(L^2(\ell\WS))},
\end{align*}
with $\gS^2(\mathfrak{H})$ the Hilbert-Schmidt norm on a Hilbert space $\mathfrak{H}$. We have
\[
    \norm{\gamma (1 - \Delta) } \le \norm{\gamma (H_c - \ri)} \norm{ (H_c + \ri )^{-1} (1 - \Delta)} \le K,
\]
for a constant $K$ independent of $c$.
We obtain an upper bound involving 
$$\norm{(1-\Delta)^{-1}V}_{\gS^2(L^2(\ell\WS))}^2. $$
We compute this Hilbert-Schmidt norm in the Fourier basis $e_k(x) := \ell^{-3/2} \re^{\ri k \cdot x}$, which gives
\begin{align*}
    \norm{(1-\Delta)^{-1}V}_{\gS^2(L^2(\ell\WS))}^2 &       = \sum_{k_1, k_2 \in \RLat / \ell} | \bra e_{k_1}, (1-\Delta)^{-1}Ve_{k_2} \ket |^2 \\
     & = \sum_{k_1, k_2 \in \RLat / \ell} \dfrac{1}{(1 + | k_1 |^2)^2} \left| \bra e_{k_1}, V e_{k_2} \ket \right|^2\\
     &=\frac1{\ell^3}\sum_{k\in\cR^*/\ell}\frac{1}{(1+|k|^2)^2}\int_{\ell\WS}V^2\\
     &\leq C\int_{\ell\WS}V^2.
\end{align*}
This concludes our sketch of the proof of~\eqref{eq:estim_Q_+-}, hence of~\eqref{eq:estim_to_be_proven_for_Q} and of Theorem~\ref{th:main_periodic}.
\end{proof}

\subsection*{Acknowledgement} This project has received funding from the European Research Council (ERC) under the European Union's Horizon 2020 research and innovation programme (grant agreement MDFT 725528 of M.L.).

\newcommand{\etalchar}[1]{$^{#1}$}


\end{document}